\documentclass{amsart}

\usepackage[all,cmtip]{xy}

\usepackage{amssymb}
\usepackage{amsmath}
\usepackage{amsthm}

\newcommand{\bd}[1]{\mathbf{#1}}
\newcommand{\sch}{\mathbf{Sch}}
\newcommand{\ssch}{\mathbf{Sch}/S}
\newcommand{\tsch}{\mathbf{Sch}/T}
\newcommand{\sschf}{(\ssch)_f}
\newcommand{\coev}{\mathrm{coev}}
\newcommand{\GL}{\mathrm{GL}}

\newcommand{\indord}{{\mathrm{Ind}}}

\newcommand{\resord}{{\mathrm{Res}}}

\newcommand{\Com}{\mathfrak{C}om}
\newcommand{\ladd}{\mathrm{add}}
\newcommand{\delh}{{\mathfrak{R}ep(\h)}}

\newcommand{\e}[1]{\mathbb{#1}}
\newcommand{\ps}{\mathrm{Ps  }}
\newcommand{\rec}{\mathrm{Rec}}
\newcommand{\kar}{\mathrm{Kar}}
\newcommand{\const}{\mathrm{Const}}

\newcommand{\injr}{\mathbf{Inj}}
\newcommand{\ev}{\mathrm{Ev}}
\renewcommand{\sp}{\mathrm{Sp}}
\newcommand{\abs}[1]{\left \lvert #1 \right \rvert}

\newcommand{\sgn}{\mathrm{sgn }}

\newcommand{\im}{\mathrm{Im}}
\newcommand{\h}{\mathscr{H}}

\renewcommand{\t}{\mathfrak{t}}
\renewcommand{\hom}{\mathrm{Hom}}

\newcommand{\gl}{\mathfrak{gl}}

\DeclareMathOperator\ind{Ind}

\DeclareMathOperator\one{One}

\DeclareMathOperator\card{Card}

\DeclareMathOperator\rep{Rep}

\DeclareMathOperator\res{Res}
\newcommand{\ress}{\res^*}
\DeclareMathOperator\inj{Inj}

\usepackage{url}

\usepackage{enumerate}

\usepackage{amsfonts}
\usepackage{verbatim}

\usepackage{amssymb}
\usepackage{mathrsfs}
\usepackage{amscd}

\newcommand{\f}[1]{\ensuremath{\mathfrak{#1}}}

\newcommand{\C}{\mathbb{C}}
\newcommand{\ob}{\mathrm{Ob}}
\newcommand{\mor}{\mathcal{M}{or}}
\newcommand{\comp}{\mathrm{Comp}}
\newcommand{\spec}{\mathrm{Spec} \ }

\newcommand{\cod}{\mathrm{cod}}
\newcommand{\dom}{\mathrm{dom}}

\renewcommand{\b}[1]{\overline{#1}}

\newcommand{\id}{\mathrm{Id}}

\title{Categories parametrized by schemes and representation theory in complex rank}
\author{Akhil Mathew}
\address{Madison High School, 170 Ridgedale Avenue, Madison, NJ 07940}
\email{akhilmathew@verizon.net}
\newcommand{\del}{\mathfrak{R}ep(S)}
\newcommand{\A}{\mathbb{A}}

\newcommand{\repord}{\rep^{\mathrm{ord}}}

\newtheorem{theorem}{Theorem}[section]
\newtheorem{lemma}[theorem]{Lemma}
\newtheorem{proposition}[theorem]{Proposition}
\newtheorem{corollary}[theorem]{Corollary}
\theoremstyle{definition}

\newtheorem{definition}[theorem]{Definition}
\newtheorem*{remark}{Remark}
\usepackage{youngtab}

\renewcommand{\subsubsection}[1]{}
\renewcommand{\paragraph}[1]{\textbf{#1}}

\usepackage{fancyhdr}
\begin{document}
\renewcommand{\headrulewidth}{0pt}
\pagestyle{fancy}

\maketitle

\begin{abstract}

Many key invariants in the representation theory of classical
groups (symmetric groups $S_n$, matrix groups $GL_n$, $O_n$,
$Sp_{2n}$) are polynomials in $n$ (e.g., dimensions of irreducible
representations). This allowed Deligne to extend the
representation theory of these groups to complex values of the
rank $n$. Namely, Deligne defined generically
semisimple families of tensor categories parametrized by $n\in \mathbb{C}$,
which at positive integer $n$ specialize to the classical
representation categories. Using Deligne's work,
Etingof proposed a similar extrapolation for many
{\it non-semisimple} representation categories built
on representation categories of classical groups,
e.g., degenerate affine Hecke algebras (dAHA).
It is expected that for generic $n\in \mathbb{C}$ such extrapolations
behave as they do for large integer $n$ (``stabilization").

The goal of our work is to provide a technique to prove
such statements. Namely, we develop an algebro-geometric
framework to study categories indexed by a parameter $n$, in which
the set of values of $n$ for which the category has a given property
is constructible. This implies that if a property holds
for integer $n$, it then holds for generic complex $n$.
We use this to give a new proof that Deligne's categories are generically semisimple. We also apply this method to
Etingof's extrapolations of dAHA, and prove that when $n$ is 
transcendental,
``finite-dimensional'' simple objects are quotients of certain standard 
induced
objects, extrapolating Zelevinsky's classification of
simple dAHA-modules for $n\in \mathbb{N}$. Finally, we
obtain similar results for the extrapolations
of categories associated to wreath products of the
symmetric group with associative algebras.
\end{abstract}

\newcommand{\ssp}{\mathbf{Sp}}
\newcommand{\unital}{\mathbf{1}}
\newcommand{\Al}{\mathbf{A}}
\section{Introduction}

The representation theory of the classical groups yields families of semisimple symmetric tensor categories $\repord(S_n), \repord(GL_n)$, etc. These categories, which depend on the parameter $n$ called the \emph{rank}, have a rich and well-studied structure, e.g., parametrizations of the simple objects and tensor product formulas. It turns out, however, that many of these parametrizations can be described polynomially in $n$.  As a result of this,  Deligne \cite{De2, De} constructed  interpolations of the categories of representations of classical groups to complex rank.  These form families of symmetric tensor categories that when specialized to positive integral parameters, reduce modulo a tensor ideal to the classical representation-theoretic categories.  These constructions, together with Deligne's proof of semisimplicity, have been extended in \cite{Kn2, Kn} to the representation categories of other families of groups such as the semidirect products $S_n \ltimes G^n$ for finite groups $G$ and $\GL_n(\mathbb{F}_q)$.     

Many families of algebraic objects of interest in  representation theory (e.g. affine Hecke algebras, Cherednik algebras, etc.) are built out of the classical groups, and thus depend on a rank $n \in \mathbb{Z}_{\geq 0}$. In many cases, the additional generators and relations of such objects can be stated in a uniform ``category-friendly'' framework independent of $n$.  This allowed Etingof to use Deligne's categories to define non-semisimple  categories of representations of algebraic objects in complex rank. 
In \cite{Et} and \cite{Etdraft}, a general program has been proposed to study
non-semisimple Deligne categories that interpolate the categories of representations of  algebraic objects derived from these groups. 
It is of interest to study the structure of these categories for their own sake; one might also hope that the theory of categorification will find in them a convenient setting.

In this paper, we begin investigating the categories of Etingof's program and show, loosely speaking, that the categories are generically similar to the classical categories.  The paper is structured as follows.  \S 2 surveys the work of Deligne and Etingof.  In \S 3, we build a method, typified by Theorem~\ref{ssgeneric} and Proposition~\ref{limisgeneric},  of deducing such properties based on a  framework of categories parametrized by schemes.  In \S 4, we apply this method to the categories of representations of the degenerate affine Hecke algebra in complex rank.  \S 5 studies the category of representations of semidirect products $S_n \ltimes A^{\otimes n}$ for $n$ complex and $A$ an associative algebra; we obtain a generic classification of simple objects. In addition, we introduce tensor structures on these categories for $A$ a Hopf algebra. In \S 6, we demonstrate that these categories are generically tensor equivalent to  Knop's categories of wreath product representations. We expect that these techniques will  be useful in the study of the other categories introduced in \cite{Et, Etdraft}, which we plan to investigate in the future.

\subsection*{Acknowledgments} The author heartily thanks Pavel Etingof and Dustin Clausen  for teaching him the subject and numerous helpful discussions and suggestions.  The author also thanks Alexander Kleshchev for helpful comments. This work was begun at the Research Science Institute at MIT, where the author was supported by the Ingersoll Rand Company.

\newcommand{\R}{\mathscr{R}}
\newcommand{\delht}{\del_t^H}

\newcommand{\Aa}{\mathscr{A}}
\newcommand{\add}{\mathrm{Add}}

\newcommand{\rrep}{\mathfrak{R}ep}
\newcommand{\rrrep}{\rrep}

\section{Preliminaries}
\subsection{Deligne's construction}
 We now quickly review Deligne's construction, which begins with a combinatorial parametrization of certain representations of the symmetric group.   
Define a \emph{recollement} of  a family of
sets $\{U_i, \ i \in A\}$ as a family of
injections $u_i\colon U_i \to C$ such that $C = \bigcup_i u_i(U_i)$.  Two recollements $C,D$ of $\{U_i\}$ are  \emph{equivalent} 
if 
there is a bijection $C \to D$ commuting with the maps $U_i \to C, U_i \to D$.
 
Now fix $n \in \mathbb{Z}_{\geq 0}$ and define   $I = \{1,
2, \dots, n\}$.  For a finite set $U$, let $ \injr(U,I)$   be the permutation representation  on the $S_n$-set $\inj(U,I)$ of injections $U \to I$.   It is easy to see that every $S_n$-representation is a direct factor of a direct sum of representations of this form.
Given an injection of finite sets $U \to V$, there is a
$S_n$-morphism $\res\colon \injr(V,I) \to \injr(U,I)$ from restriction and an
adjoint $\ress\colon \injr(U,I) \to \injr(V,I)$ with respect to the natural  bilinear forms on $\injr(U,I),\injr(V,I)$, i.e., the ones making the standard bases $\{e_f , f \in \inj(U,I)\}$ and $\{e_g , g \in \inj(V,I)\}$ orthonormal. 

Let $\rec(U,V)$ denote the set of equivalence classes of recollements of $U,V.$  For  $C \in \rec(U,V)$  with representative maps $u\colon U \to C ,v\colon V \to C$,  define  \( (C) = \res_v \circ \ress_u \in \hom_{S_n}( \injr(U,I),\injr(V,I)).\) Another way to think of $(C)$ is as the $S_n$-invariant element of $\injr(U,I) \otimes \injr(U,V)$ consisting of the sum of all basis elements $e_f \otimes e_g$ for $f,g$ inducing the recollement $C$ of $U,V$. Hence, there is a map from $\C\left< [\rec(U,V)] \right>$ to  $\hom_{S_n}(\injr(U,I),\injr(V,I))$, where for a finite set $S$, we let $\C\left<S \right>$ be the free vector space on the basis $S$.  By  Proposition 2.9 of \cite{De}, this map is always surjective and, for $n>\abs{U}+\abs{V}$,  an isomorphism.  It turns out that composition can be described via polynomials in $n$ and recollements.  More precisely,   for finite sets $U, V, W$ and recollements $C \in \rec(U,V), D \in \rec(V,W), E \in \rec(U,W)$, there exists $P^E_{C,D}(T) \in \mathbb{Z}[T]$ such that
\begin{equation} (D)\circ (C) = \sum_{E \in \rec(U,W)} P^E_{C,D}(n) (E) \in \hom_{S_n}( \injr(U,I), \injr(W,I)) \quad \forall \ D,C,n .\end{equation}

 Deligne thus  defined  a $\mathbb{Z}[T]$-linear category $\rep'(S_T)$ as follows:
$\rep'(S_T)$ contains objects $[U]$ for finite sets $U$,
$\hom([U],[V])$ is   $\e{Z}[T]$-free on $\rec(U,V)$, and composition is $\e{Z}[T]$-linear with
\(  (D) \circ (C) = \sum_E
P^E_{C,D}(T) (E). \) For a ring $R$ and $t \in R$, specializing  $T \to t$ gives categories $\rep'(S_t, R)$.  One then 
takes the  pseudo-abelian envelope to obtain the category $\rep(S_t, R)$, abbreviated $\rep(S_t)$ when the
context is clear and $R$ is a field.
 
$\rep(S_t,R)$ can be made into a symmetric $R$-linear monoidal category.  The monoidal structure comes from $[U] \otimes [V] = \bigoplus_{C \in \rec(U,V)} [C]$, which extends to the pseudo-abelian envelope; it can be checked that this interpolates the usual structure.
The monoidal structure on morphisms is given in an analogous manner.  See \cite{De}, Proposition 2.8.  In this category, in fact, each object is self-dual.

In stating the next result, it will be convenient to introduce full subcategories $\rep(S_t, R)^{(N)}$ defined in the same way, but with the sets $U$ used restricted to have cardinality at most $N$.  

\begin{theorem}[Deligne \cite{De}, Prop. 5.1] Let $R$ be a field of characteristic zero. \begin{enumerate}
\item When $t \notin \{0, 1, \dots, 2N-2\}$, $\rep(S_t,R)^{(N)} $ is semisimple, and  the simple objects are indexed by the  partitions of integers $\leq N$.  
\item If $t \in \mathbb{Z}_{\geq 0}$ with $t > 2N$, then $\rep(S_t,R)^{(N)}$ is a full subcategory of the ordinary category $\repord(S_t,R)$ of finite-dimensional $R$-linear representations of $S_t$ (via $[U] \to \injr(U,I)$). 
\item If $t \in \mathbb{Z}_{\geq 0}$, and $\mathcal{N}$ is the tensor ideal of negligible morphisms in $\rep(S_t, R)$, then
\( \repord(S_t, R) \simeq \rep(S_t,R)/\mathcal{N} \)
as symmetric tensor categories.
\end{enumerate}
\end{theorem}

 One can   think of this phenomenon as follows.  Consider a Young diagram of size $M$:
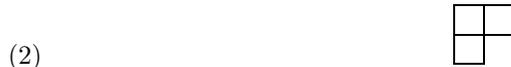
\begin{equation} \label{smallyng}
\yng(2,1)
\end{equation} 
One can think of the simple object in $\rep(S_t, R)^{(N)}$ indexed by it as corresponding to the ``Young diagram''  of ``size $t$'' obtained by adjoining a very long line of ``length'' $t-M$ at the top:
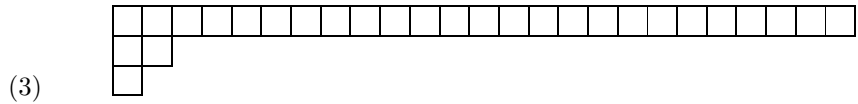
\begin{equation} \label{bigyng}
\yng(25,2,1)
\end{equation}
More precisely, in case (2) of the theorem, the functor sends the simple object associated to \eqref{smallyng} to the Specht module over $S_t$ associated to the (actual) Young diagram \eqref{bigyng}.  In particular, $\rep(S_t, R)^{(N)}$ is equivalent to the subcategory  $\repord(S_t, R)^{(N)}$ of $\repord(S_t,R)$ spanned by the Specht modules associated to Young diagrams with at most $N$ squares below the first row.

\subsection{The dAHA}
The degenerate affine Hecke algebra  (dAHA)  of type $A$ is the associative unital algebra $\h_n$ over $\C$  generated by  variables
$s_i$ corresponding to the
transpositions $(i \ i+1) \in S_n$ for $1 \leq i <n$, and commuting variables $x_1, \dots,
x_n$, with relations $s_i s_j = s_j s_i \text{\  if \ } \abs{i-j} \geq 2;   \ (s_i s_{i+1})^3 = 1;  \  s_i^2 = 1;  s_j x_i = x_i s_j  \text{ if }   j \neq i-1, i; $ and $   s_i x_i - x_{i+1} s_i = 1.$
Another presentation of the dAHA can be given via generators  $\sigma \in S_n$ and  $y_1,
\dots , y_n$ with the relations
\begin{equation}
\sigma y_i \sigma^{-1} = y_{\sigma(i)} , \ \sigma \in S_n \quad \text{and} \quad [y_i, y_j] = \frac{1}{4} \sum_{k \neq i,j} (ijk) - (jik) \ \text{if} \  \ i \neq j.   \end{equation}  
The two
presentations can be shown   equivalent via  the transformation
\( 
s_i \to (i \ i+1),  \  x_i \to y_i
  - (1/2) \sum_{l=1}^n \sgn(i-l) (i \ l), \) where
$\sgn \ x = x/\abs{x}$ if $x \neq 0$ and zero otherwise.  See \cite{Kl}.

\subsection{The categories $\rep(\h_t)$}

The definition of a $\h_n$-module can be rewritten in
terms of $S_n$-modules, the \emph{permutation representation} $\f{h} = \injr(\one, I)$ for a fixed one-element set $\one$, and the central \emph{Jucys-Murphy element} \(
\Omega = \sum_{i<j} (i \ j) \in \mathrm{Cent} \ \C[S_n]. \)  There is a corresponding endomorphism $\Omega$ of the functor
$\id_{\repord(S_n)}$.
Fix  $M \in \repord(S_n)$.  On  $\f{h} \otimes \f{h}
\otimes M \in \repord(S_n)$, let $P_{1,2} = \sigma_{\f{h}, \f{h}} \otimes \id_M$ where  $\sigma_{\f{h},\f{h}}\colon \f{h} \otimes \f{h} \to \f{h} \otimes \f{h}$ is the permutation morphism. Define the endomorphisms of $\f{h} \otimes \f{h}
\otimes M$:
\( \Omega^{2,3} = \id \otimes \Omega_{\f{h} \otimes M} \) and \(
\Omega^{1,3} =P_{1,2} \circ  \Omega^{2,3}
\circ  P_{1,2} .\)
Also, let $B_{\f{h}}\colon \f{h} \otimes \f{h} \to \unital$ be the evaluation map.
Then:
\begin{proposition}[Etingof \cite{Et, Etdraft}] To give a module over  $\h_n$ is equivalent to giving an $S_n$-representation $M$
with a $S_n$-morphism $y\colon \f{h} \otimes M \to M$ satisfying
\begin{equation} \label{eq: dAHA relation} y \circ (\id \otimes y) \circ (\id - P_{1,2})  = (B_{\f{h}} \otimes
  \id) \circ [\Omega^{1,3}, \Omega^{2,3}] \colon \f{h} \otimes \f{h} \otimes M \to M.
\end{equation}
A morphism between $\h_n$-modules is a $S_n$-morphism commuting with the $y$-maps.
\end{proposition}

The terms $P_{1,2}, B_{\f{h}}, \Omega^{1,3}, \Omega^{2,3}$ in \eqref{eq: dAHA relation} can be interpolated to $t \notin \mathbb{Z}_{\geq 0}$ because Deligne's categories are symmetric monoidal categories, and each object is self-dual.  The object $\f{h}$ is defined as $[\one]$. The interpolation of the Jucys-Murphy endomorphism   is due to   \cite{ComesOstrik} and will be discussed below.

Note that where \cite{Etdraft} uses $\widetilde{\rep}(S_t)$, we use $\rep(S_t)$; our definitions are thus slightly different, but they coincide outside of $\mathbb{Z}_{\geq 0}$.  Our convention will be more convenient for the method developed in this paper.

 \begin{definition}[Etingof \cite{Et, Etdraft}]
 For $t \in \mathbb{C}$,
objects in $\rep(\h_t)$ are  objects $M$    in Deligne's category
$\rep(S_t)$ equipped with $\rep(S_t)$-morphisms $y \colon \f{h} \otimes M \to M$
satisfying \eqref{eq: dAHA relation}.  Morphisms in $\rep(\h_t)$ are morphisms in $\rep(S_t)$ commuting with the $y$ maps. 
\end{definition}

\begin{remark} {The definition in \cite{Et, Etdraft}
  actually used ind-objects.  This
  suggests further research, as we only study the finite-dimensional theory. }
\end{remark}
\newcommand{\ord}{\mathbf{Poset}}
\newcommand{\bool}{\mathbf{Bool}}

\newcommand{\schn}{\sch_{\mathscr{N}}}

\section{Categories depending on a parameter}
\label{secpar}
\subsection{Definitions}
First, we review some relevant algebro-geometric facts.   
Let $\schn$ denote the category of noetherian schemes.   
 For $X \in \schn$, let $\const(X)$ denote the collection of its  {constructible}  subsets.   If $f\colon X \to Y$ is a morphism, then there exists $f^{-1}\colon \const(Y) \to \const(X)$ sending $C \to f^{-1}(C)$; this makes $\const$ a contravariant functor $\schn \to \bool$, where $\bool$ is the category of Boolean algebras.  Similarly, by a special case  of Chevalley's theorem,   if $S \in \schn$ and $\sschf$ denotes the category of $S$-schemes of finite type, we can make $\const$ a covariant functor $\sschf \to \ord$, for $\ord$ the category of posets, where $f\colon X \to Y$ defines the map  $f_{*}\colon \const(X) \to \const(Y)$, $D \to f(D)$; by abuse of notation, we often write $f$ for $f_{*}$.  

Let $X,Y,T$ be $S$-schemes and $X_T = X \times_S T, Y_T = Y \times_S T$.  The diagram
\[ \begin{CD}
X_T @>{p_X}>> X \\
@V{f_T}VV @V{f}VV \\
Y_T @>{p_Y}>> Y  \end{CD}
\]
is cartesian, so for any $C \subset X$, we have
$ p_Y^{-1}(f(C)) = f_T(p_X^{-1}(C))$ 
by \cite{EGA}, I.3.4.5.

  Given a scheme $X$, $G \subset X$, and a point $s \in S$, we define $G_{\overline{s}} = {f'_{\b{s}}}^{-1}(G) \subset X_{\b{s}}$, where $f_{\b{s}}\colon \spec \overline{k(s)} \to S$ is the canonical morphism and ${f'_{\b{s}}}\colon X_{\b{s}} = X \times_S \spec \overline{k(s)} \to X$ is the base extension to $X$. Here $\overline{k(s)}$ is an algebraic closure of $k(s)$;   when used, the choice of  $\overline{k(s)}$ will be irrelevant.

\newcommand{\cat}{\mathbf{Cat}}

\begin{definition} A \emph{category over $S$}  is a category internal to the category $\ssch$ of $S$-schemes and $S$-morphisms.  There is a notion of a \emph{functor over $S$}, or even a \emph{natural transformation over $S$}, so we may consider the 2-category $\cat_S$  of categories over $S$. 
\end{definition}

A category over $S$ is thus a four-tuple $(\ob, \mor, \comp, \id)$, where $\ob$ is a $S$-scheme corresponding to the objects; $\mor$ is a $S$-scheme with a $S$-morphism $\mor \to \ob \times_S \ob$, whose first and second projections are denoted  $\dom$ and $\cod$ respectively; $\comp$ is a $S$-morphism $\mor \times_\ob \mor \to \mor$; and $\id\colon \ob \to \mor$ is a $S$-morphism.  There are certain conditions on these that need to be satisfied, cf.    \cite{MacLane}.  A category over $S$ is said to be of  \emph{finite type} when the morphisms $\ob \to S$, $\mor \to S$ are of finite type, or if it is internal to $\sschf$.

Given a category $\f{C}$ over $S$,  one may construct a category $\mathcal{C}$ by the Yoneda process.  The objects are the $S$-sections $S \to \ob$; the morphisms from $a\colon S \to \ob, b\colon S \to \ob$ are defined as the $S$-maps $h\colon S \to \mor$ such that $\dom \circ h = a, \cod \circ h = b$.  Composition of appropriate $h\colon S \to \mor, g\colon  S \to \mor$ is defined as $\comp \circ (h \times_S g)$.
There is thus a 2-functor $\R_S\colon \cat_S \to \cat$, where $\cat$ is the 2-category of categories, which we often abbreviate by $\R$. 
 
Next, since the functor $\ssch \to \tsch$ defined by $X \to X \times_S T$ preserves finite limits, we can perform \emph{base extension} of $S$-categories; hence, given $T \to S$, there is a 2-functor $B_S^T\colon \cat_S \to \cat_T$.   We abbreviate $\f{C}_T = B_S^T(\f{C})$, $\mathcal{C}_T = \R_T(B_S^T(\f{C}))$.

\subsection{Constructibility of categorical properties}
   For a $S$-scheme $Y$, by abuse of notation denote by $\Delta_Y \subset Y \times_S Y$ the locally closed subscheme that is the image of the diagonal map $\Delta_Y\colon Y \to Y \times_S Y$.  
Let $\f{C} = (\ob, \mor, \comp, \id)$ be a category  of finite type over the noetherian scheme $S$.  
 Consider the $S$-scheme $C = (\mor \times_{\ob \times_S \ob} \mor) \times_{\ob} \mor$, where the map $(\mor \times_{\ob \times_S \ob} \mor) \to \ob$ comes from $\cod$, and $\mor \to \ob$ is $\dom$.  There is  a map $f \colon C \to \mor \times_{\ob \times_S \ob} \mor$ obtained by composition on each factor: $f = (\comp \circ (p_3, p_1), \comp \circ (p_3, p_2))$, where $p_1, p_2,p_3 $   are the projections on the respective factors  of $C$.  Now consider $f^{-1}(\Delta_{\mor}) - (\Delta_{\mor} \times_{\ob} \mor) \subset C$; this is a constructible set $X$.  Set  $Mon = \mor - p_3(X) \subset \mor$.  The formation of $Mon$ commutes with base extension, and $Mon$ is constructible by Chevalley's theorem.  Finally, when $S = \spec k$ for $k$ an algebraically closed field, then the closed (i.e., $k$-valued) points of $Mon$ correspond precisely to monomorphisms in the category $\mathcal{C}_k = \mathcal{C}_{\spec  k}$.  
This discussion and its dual imply the following result.

\begin{proposition} \label{injisgeneric} Suppose the   scheme $S$ is  noetherian  and $\f{C}$ is of finite type.  Then there is a constructible set $Mon \subset \mor$ (resp. $Epi \subset \mor$) such that for each $s \in S$,   $Mon_{\overline{s}} \subset \mor_{\b{s}}$ (resp. $Epi_{\b{s}} \subset \mor_{\b{s}}$) has closed points corresponding to the monic (resp. epic) maps in $\mathcal{C}_{\b{s}} = \mathcal{C}_{\b{k(s)}}$.
\end{proposition}

\newcommand{\lcat}{\mathbf{AbCat}}

\subsubsection{Ab-categories over $S$}
Next, we study  additive structures on $S$-categories.

\begin{definition} An \emph{Ab-category over $S$} is a $S$-category together with the structure of a  $\ob \times_S \ob$ group scheme on $\mor$, with an addition \(\ladd\colon \mor \times_{\ob \times_S \ob} \mor \to \mor\) that commutes with composition.  
   An \emph{additive functor over $S$} is one which commutes with the group structures.   There is thus a 2-category $\lcat_S$ of Ab-categories over $S$, additive functors, and natural transformations.  Base extension preserves additivity. 
\end{definition}

The 2-functor $\R\colon \cat_S \to \cat$  defined earlier induces a 2-functor from $\lcat_S$ to $\lcat$ for $\lcat$ the 2-category of Ab-categories, additive functors, and natural transformations. 

If $A$ is an object in an ordinary Ab-category $\mathcal{C}$, say that $A$ is \emph{simple} if every map $A \to B$ is either zero or monic, or equivalently if any map  $C \to A$ is either zero or epic, i.e., if for all $B,C \in \mathcal{C}$ and nonzero $f\colon C \to A$ and $g\colon A \to B$, the composition $g \circ f$ is nonzero.

\begin{proposition} \label{simplicityisgeneric} \label{simpleisgeneric} Suppose   $S$ is  noetherian   and $\f{C}$ is an Ab-category of finite type over $S$.  Then there is a constructible set $Sim \subset \ob$ such that for each $s \in S$,   $Sim_{\overline{s}} \subset \ob_{\overline{s}}$ has closed points corresponding precisely to the simple objects in $\mathcal{C}_{\b{s}}$.
\end{proposition}
\begin{proof} Let $z\colon \ob \times_S \ob \to \mor$ be the section corresponding to the zero morphism.  Consider the constructible set $R = \mor - Mon - z(\ob \times_S \ob) \subset \mor$; $R$ loosely speaking parametrizes nonzero non-monomorphic morphisms.  Set $Sim = \ob - \dom(R)$, and one may argue as in the discussion preceding Proposition~\ref{injisgeneric}.  
\end{proof}

 {A simple example of this phenomenon is as follows.  Fix a finitely generated but not necessarily commutative $k$-algebra $A$ for $k$ an algebraically closed field and a cut-off $N \in \mathbb{Z}_{\geq 0}$.  Then the category of $k$-valued matrix representations of $A$ of dimension at most $N$  and intertwining maps can be obtained by applying   $\R_k$ to  a category of finite  type over $\spec k$. These propositions then reduce to   well-known facts, e.g., that injective (or surjective) homomorphisms are described by a constructible set.  Indeed, the set of injective homomorphisms is even open by linear algebra, because it can be described by the nonvanishing of certain minors.  One can also apply these results to the categories of finite-dimensional representations of  algebras depending on a parameter, such as quantum $\mathfrak{sl}_2$, or affine or finite Hecke algebras; these form categories over open subsets of $\A_{\C}^1$.}

\subsubsection{Universals, limits, and colimits}
Given a scheme $X$, let $cl(X)$ denote the subspace of closed points.  The following lemma is elementary.

\begin{lemma} \label{bijisconst} Let $f\colon X \to Y$ be a $S$-morphism for $X,Y$ of finite type over the noetherian scheme $S$.  Then the set  of $s \in S$ with the map $cl(f_{\b{s}})\colon cl(X_{\b{s}}) \to cl(Y_{\b{s}})$ bijective is constructible.
\end{lemma}

\begin{proposition} \label{univisgeneric} Let $\f{C} = (\ob, \mor, \comp, \id)$ be a $S$-category of finite type, with $S$ noetherian.  There is a constructible set $In \subset \ob$ (resp. $Fi \subset \ob$) such that, for $s \in S$, the closed points in $In_{\b{s}}$ (resp. $Fi_{\b{s}}$) correspond to the initial (resp. final) objects in $\mathcal{C}_{\b{s}}$.
\end{proposition}

\begin{proof} We treat the initial case.
An object $A$ in a category $\mathcal{C}$ is initial if and only if for all $B \in \mathcal{C}$, the diagonal map $\hom_{\mathcal{C}}(A,B) \to \hom_{\mathcal{C}}(A,B) \times \hom_{\mathcal{C}}(A,B)$ is bijective.
 Hence, consider $\Delta_{\mor}\colon \mor \to \mor \times_{\ob \times_S \ob} \mor$.  Let $T \subset \ob \times_S \ob$ denote the constructible set of $t \in \ob \times_S \ob$ such that $(\Delta_{\mor})_{\b{t}} = \Delta_{\mor_{\b{t}}}$ is bijective.
Then let $In = \ob - p_1(\ob \times_S \ob - T)$, where $p_1\colon \ob \times_S \ob \to \ob$ is projection on the first factor.
\end{proof}

Given a $S$-category $\f{C}$ and a finite (ordinary) category $\mathcal{Z}$, we can consider the functor category  $\f{C}^{\mathcal{Z}}$, whose object scheme  $E$ is constructed so as to parametrize $\mathcal{Z}$-diagrams of morphisms, and whose morphism scheme is constructed to parametrize morphisms of diagrams.  Specifically, $E$ is constructed from the scheme  $\prod_{f\colon A \to A', \ A, A' \in \mathcal{Z}} \mor$ via a suitable intersection of equalizers corresponding to composability and  the relations defining $\mathcal{Z}$, for instance. 
Finally, we can consider the comma category  $\Com$ as a category over $E$; indeed, there is a functor over $S$ from $\Com$ to $\f{C}^{\mathcal{Z}}$.  Let $F$ be the object scheme.  Then, since the universal objects in a comma category consist of diagram schemes and their limits,  Proposition~\ref{univisgeneric} implies the following corollary.  It can also be dualized for colimits.

\begin{corollary} \label{limisgeneric} Notation as above, if $\f{C}$ is of finite type and $S$ noetherian, there is  a constructible subset $L \subset F$ such that for each $s \in S$, $L_{\b{s}}$ has closed points corresponding to   pairs of $\mathcal{C}_{\b{s}}$ valued  $\mathcal{Z}$-diagram schemes   with limits in $\mathcal{C}_{\b{s}}$, and those limits.
\end{corollary}


\subsection{Envelopes of categories}
{Recall that the \emph{Karoubi envelope} of a category is formed by adding formal images of idempotent elements.}  The construction, presented, e.g., in \cite{NatanMorrison}, is as follows.  Given $\mathcal{C}$, an object in $\kar(\mathcal{C} )$ is a pair $(A, e)$ where $A \in \mathcal{C}$ and $e \in \hom_{\mathcal{C}}(A,A)$ is idempotent.  A morphism between $(A,e)$ and $(B,f)$ is a $\mathcal{C}$-morphism $h\colon A \to B$ with $h = h \circ e = f \circ h$.  Note incidentally that the identity endomorphism of $(A,e)$ is given by $h=e$.  The Karoubi envelope of an Ab-category is an Ab-category too.  An Ab-category equivalent to its Karoubi envelope, or equivalently such that every idempotent has an image, is called \emph{Karoubian.}
Next, the \emph{additive envelope} of an Ab-category $\mathcal{A}$ has objects consisting of formal direct sums $\bigoplus_{i=1}^n A_i$ with $A_i \in \mathcal{A}$, with morphisms between  $\bigoplus_{i=1}^n A_i,  \bigoplus_{j=1}^m B_j$ defined as appropriate $m$-by-$n$ matrices, and composition as matrix multiplication.
Finally, the \emph{pseudo-abelian envelope} is the Karoubi envelope of the additive envelope.  A \emph{pseudo-abelian} category is an additive category where every idempotent endomorphism is a projection, or alternatively one equivalent to its pseudo-abelian envelope.

\begin{proposition} \label{karoubiandps}
   There is a 2-functor $\kar \colon \lcat_S \to \lcat_S$ preserving the property of finite type such that $\R(\kar(\f{C}))$ is equivalent to the Karoubi envelope of $\R(\f{C})$.  There are similar 2-functors $\add, \ps \colon \lcat_S \to \lcat_S$, though these do not necessarily preserve the property of finite type.
$\kar, \add, \ps$ all commute with base extension.
\end{proposition}

As the proof will show, this is true for internal categories in a category with arbitrary fibered products.  Even so, we cannot find this result in a reference.

\begin{proof} All these categories are constructed to parametrize suitable diagrams. Given $\f{C} = (\ob, \mor, \comp, \id)$, we take for the object scheme $\ob_K$ of $\kar(\f{C})$ the intersection   of the equalizer of $\ob \times_S \mor$ under the three maps $\{p_1, \dom \circ p_2 , \cod \circ p_2 \}$ and the equalizer of the maps $\{\comp \circ \Delta_{\mor} \circ p_2, p_2\}$, where $p_1, p_2$ are projections on the first and second factor.  The morphism scheme is taken to be the intersection of the equalizers of 
\( \ob_K \times_S \ob_K \times_S \mor  \)
under $ \{  p_1 \circ q_1, \dom \circ q_3\}, \{  p_1 \circ q_2, \cod \circ q_3\}, \{\comp \circ (q_3, p_2 \circ q_1), \comp \circ (p_2 \circ q_2, q_3), q_3 \}$, where $q_1, q_2, q_3$ are the projections from \( \ob_K \times_S \ob_K \times_S \mor  \) onto the first, second, and third factors.
Then the first part of the result follows from the definitions, since base extension commutes with $S$-products.  The proof for $\add$ is similar, though because of the infinite coproduct involved over direct sums of larger sizes, $\add$ does not preserve the property of finite type.  Then $\ps = \kar \circ \add$.
\end{proof}

\subsection{Semisimplicity} We shall now demonstrate a method of proving generic semisimplicity of a family of categories.

\begin{theorem} \label{ssgeneric} Let $\f{C}$ be a finite type $Ab$-category over $S$ for  $S$ a  noetherian scheme.  Suppose $Z \subset S$ is a Zariski dense subset such that for $s \in Z$ we have $\ps( \mathcal{C}_{\b{s}})$ semisimple abelian with endomorphism rings either zero or infinite.  Then,  for $\xi$ in a dense open subset of $S$,   $\ps(\mathcal{C}_{\b{\xi}})$ is semisimple abelian.
\end{theorem}

\begin{proof}
Without loss of generality, assume $\mathcal{C}_{\bar{s}}$ Karoubian for $s \in Z$.  Recall that $Sim \subset \ob$ is constructible, for $\ob$ the object scheme of $\f{C}$.    By Propositions~\ref{simplicityisgeneric} and \ref{limisgeneric}, there are constructible subsets $Sim^1, Sim^2, \dots \subset \ob$ with $(Sim^n)_{\b{s}}$ having closed points corresponding to objects in $\mathcal{C}_{\b{s}}$ which are the direct sum of at most $n$ simple objects, for each $s,n$.  

Let $M = \sup_{x \in \ob \times_S \ob} \dim(\mor_x) = \sup_{x \in \ob \times_S \ob} \dim(\mor_{\b{x}})$.  We show that for $s \in Z$, any object in $\mathcal{C}_{\b{s}}$ is a direct summand of at most $\sqrt{M}$ simple objects.  Indeed, if $X \in \mathcal{C}_{\b{s}}$ is nonzero and simple, then $\dim \mor_{\b{(X,X)}}   \geq 1$ since  $\hom(X,X)$ is infinite.   If $Y = X_1 \oplus \dots \oplus X_m$ for $X_i$ nonzero and simple in $\mathcal{C}_{\b{s}}$, then \(m^2 \leq \dim \mor_{\b{(Y,Y)}} \leq M,\) since  from projection, inclusion, and $\comp$, there is an isomorphism of schemes $\mor_{{\b{(Y,Y)}}} \simeq \prod_{i, j \ \ \overline{k(s)}} \mor_{\overline{ (X_i,X_j)}}$.  In particular, we see that $(\ob - Sim^N)_{\b{s}} = \emptyset$ if $N > \sqrt{M}$ and $s \in Z$.  Hence $\ob_{\b{\xi}} = (Sim^N)_{\b{\xi}}$ for $\xi \in U$ for $U \subset S$ open and dense.  By shrinking $U$ if necessary and using Proposition~\ref{limisgeneric}, we may assume $\mathcal{C}_{\b{\xi}}$ admits kernels and cokernels for $\xi \in U$.  The additive envelope of $\mathcal{C}_{\b{\xi}}$ for $\xi \in U$ is thus semisimple abelian and consequently is  equivalent to the pseudo-abelian envelope.
\end{proof}

\subsection{Deligne's categories as $\A_{\C}^1$-categories}

We now incorporate Deligne's construction into the above framework.
It follows that there is a category $\rrep'(S_T)$ over $\A^1$ with $\R(\rrep'(S_T)) \simeq \rep'(S_T)$; the object scheme is an infinite coproduct $\coprod_{j \in \mathbb{Z}_{\geq 0}} \A^1$ while $\mor = \coprod_{j,k} \A^1 \times \A^{\rec(j,k)}$, where $\rec(j,k)$ denotes the number of equivalence classes of recollements of a set with $j$ elements and one with $k$ elements.  

Let $\rrrep''(S_T) = \ps( \rrep'(S_T))$; this admits a bifunctor over $\A^1$ corresponding to the tensor product.  By base extension of $\rrrep''(S_T)$ via $\mathbb{Z}[T] \to R$, $T \to t$, we find: 

\begin{proposition} There is a  $\spec R$-category $\del^R_t$ with $\R(\del^R_t) \simeq  \rep(S_t,R)$.  
\end{proposition}

\subsubsection{Tensor structure}
It is clear that the tensor structure on $\del^R_t$ is a bifunctor $T$ over $\spec R$.  Moreover, if $D$ is the ``diagonal functor''  $\del^R_t \to \del^R_t \times_R \del^R_t$, then there are natural transformations over $\spec R$,
\( \id \to B(D), \ B(D) \to \id \)
from the self-duality in $\rep(S_t,R)$, which satisfy the appropriate axioms for evaluations and coevaluations.

\subsubsection{Filtration}
We let $\del = \del^{\C[T]}_T$; this is a category over $\A_{\C}^1$.
For $t \in \A_{\C}^1$, let $\del_t$ be the fiber of $\del$ at $t$.  There is a similar situation for the $\del^{(N)}$, which correspond to the categories $\rep(S_t)^{(N)}$ defined in 4.1 of \cite{De}; these are pseudo-abelian envelopes of $\A_{\C}^1$-categories of finite type.  Finally, we let $\del^{(M,N)}$ correspond to the category $\rep(S_t)^{(M,N)}$, which we define as the Karoubi envelope of the category obtained by formal direct sums with at most $M$ elements of objects in $\rep'(S_t)^{(N)}$; then $\del^{(M,N)}$ is an approximation to $\del^{(N)}$ and, most importantly, is of finite type. 

\subsubsection{Other Deligne categories}
There is a similar situation for the interpolations of $\repord({GL}_n), \repord({O}_n),$ and  $\repord(Sp_{2n})$, in view of their definitions in \cite{De2, De}.  We omit the details.  
We thus obtain a new proof of the following result.

\begin{theorem}[Deligne \cite{De}, contained in Theorem 2.18] Let $F$ be an algebraically closed field of characteristic zero.  Then $\rep(S_t, F)^{(N)}$ is semisimple abelian for all but finitely many $t \in F$, which are algebraic over $\mathbb{Q}$.  $\rep(S_t, F)$ is semisimple abelian for transcendental $t$.  Similarly for the other Deligne categories associated to the general linear, orthogonal, and symplectic groups.
\end{theorem}

\begin{proof} The statement about $\rep(S_t, F)^{(N)}$ follows from Theorem~\ref{ssgeneric} applied to $(\del_T^{F[T]})^{(1,N)}$ and Maschke's theorem. 
The statement about algebraicity follows because all the schemes in question are actually defined over $\mathbb{Q}$.
This implies the rest of the corollary for $\rep(S_t,F)$ since $\rep(S_t, F)^{(N)} \subset \rep(S_t,F)$ is a full subcategory and $\rep(S_t,F) = \bigcup_N \rep(S_t,F)^{(N)}$.   The other cases are handled similarly. \end{proof}

Unfortunately, this method does not address the rationality questions that arise when $F$ is not algebraically closed.  For a different method that yields the generic result using the specific properties of the partition algebra, see \cite{ComesOstrik}.  For results on the generic semisimplicity of \emph{tensor} categories obtained using the calculus of relations on regular categories, cf. \cite{Kn}; this method uses the specific properties of degree functions.

\subsubsection{Specialization functors}

By base extension, there is a \emph{specialization functor}
\[ \sp_t\colon \rep(S_T, \C[T])^{(M,N)} \simeq \R_{\A_{\C}^1}(\del^{(M,N)}) \to \rep(S_t)^{(M,N)} \] for $t \in \C.$  In simple terms, this corresponds to substituting a value in for a polynomial.  When $n \in \mathbb{Z}_{\geq 0}$, we may also define $\ssp \colon \rep(S_T, \C[T]) \to \repord(S_n)$.

For a given object or morphism $a$ in $ \rep(S_T, \C(T))$, we can define almost all specializations $\sp_t a$.  Indeed, if $Y$ is either the object or morphism scheme, then the map $\spec \C(T) \to Y$ factors into a map of the form
\(
\spec \C(T) \to \spec \C[T]_{P(T)} \to  Y,\)
where $P(T)$ is a suitable polynomial.

\subsubsection{Properties of specialization}  For almost all $t$, specialization sends the simple objects $X_\pi$ in Deligne's category $\rep(S_T, \C(T))^{(N)}$ indexed by the partition $\pi \colon l_2 + l_3 + \dots + l_k = l$ to those of $\rep(S_t)^{(N)}$ similarly indexed, by the proofs in \cite{De}, \S 5.   Deligne constructs the relevant objects by using certain ``Young projectors'' on objects $[U]$.

Finally, for  $n > 2N+1$, the simple object $X_\pi$ of $\rep(S_n)^{(N)}$, which is a full subcategory of $\repord(S_n)$, corresponds to the classical Specht module 
associated to the partition $n=n - l + l_2 + \dots + l_k$, in view of 6.4 in \cite{De}.

\section{The dAHA categories}

We shall now apply the methods of the previous section to the Etingof categories associated to the dAHA.


\subsection{$\rep(\h_t)$ in the present framework} 

It is not immediate that the Etingof categories form an instance of ``categories depending on a parameter.'' This follows from the next result, essentially due to Comes and Ostrik.

\begin{proposition}[Comes and Ostrik \cite{ComesOstrik}] \label{interpolateJM} On the $\A_{\C}^1$-category $\del$, there exists an endomorphism $\Omega$
  of the identity functor $\id_{\del}$ interpolating the action of the Jucys-Murphy
  element in the integer case.
\end{proposition}  

It is thus similarly possible to define specialization functors in this setting.

Consequently, there is an $\A_{\C}^1$-category $\delh$ with 
\( \R( \delh_t) \simeq \rep(\h_t)\).  There is a similar situation for $\delh^{(N)}$ and $\delh^{(M,N)}$, which correspond to the categories where the underlying object in $\rep(S_t)$ belongs to $\rep(S_t)^{(N)}$ (resp. $\rep(S_t)^{(M,N)}$).

\subsection{The classical dAHA representation theory}
We state the classification of simple
 $\h_n$-modules.  Given $a \in \C$, there exists an \emph{evaluation homomorphism} $\h_n \to \C[S_n]$
(\cite{AS}, Lemma 1.5.3)  that is the identity on $\C[S_n]$ and sends $x_i \to a+  \sum_{j<i}
(i \ j)$. Alternatively, the homomorphism sends $y_i \to a + \frac{1}{2}
\sum_{j \neq i} (i \ j)$.  There is a
corresponding evaluation functor $\ev_a \colon \repord S_n \to \repord \h_n$. 
Next, there is a homomorphism 
\(\h_m \otimes
  \h_n \to \h_{m+n} \) inducing the canonical maps $\C[S_m] \otimes \C[S_n] \to \C[S_{m+n}]$ and $\C[x_1, \dots, x_m] \otimes \C[x_{m+1}, \dots, x_{m+n}] \to \C[x_1, \dots, x_{m+n}]$.  
They lead to induction functors $\repord \h_m \times \repord \h_n
\to \repord \h_{m+n}$ sending $A,B$ to $\h_{m+n} \otimes_{\h_m \otimes
  \h_n} ( A \otimes B )$, which can be iterated.   Cf. \cite{Kl}.

\subsubsection{Standard modules}

Fix an unordered partition $\pi\colon
n= l_1 + \dots + l_k $.   Let $\t_k = \C^k$, spanned by a basis
$\epsilon_1^\vee, \dots, \epsilon_k^\vee$; this corresponds to the abelian Lie subalgebra of $\gl_k$ consisting of diagonal matrices.
Let $\t_k^*$ be its dual, and $\epsilon_1, \dots, \epsilon_k \in \t_k^*$ be the
dual basis to $\epsilon_1^\vee, \dots, \epsilon_k^\vee$.  Set $\rho =
\frac{1}{2} \sum_{i<j} (\epsilon_i - \epsilon_j) .$ Let
\[ D_k = \{ \lambda \in \t_k^*\colon (\lambda + \rho)(\epsilon_i^\vee -
\epsilon_j^\vee) \notin \mathbb{Z} \cap (-\infty,0), \ 1 \leq  i < j \leq
k\} .\]

Fix $\mu = \sum_i \mu_i \epsilon_i \in \t_k^*$.  Define \[M^n(\mu)_\pi = \ind_{{l_1} ,   \dots , { l_k}}^{n} ( \ev_{\mu_1 +\rho_1} \unital, \dots,  \ev_{\mu_k + \rho_k}
\unital) \in \repord(\h_n).\]
If $\mu + l_1 \epsilon_1 + l_2 \epsilon_2 + \dots l_k \epsilon_k \in
D_k$, then $M^n(\mu)_\pi$ is called a \emph{standard module}.

\begin{theorem}[Zelevinsky \cite{Ze}]  \label{zv}
\begin{enumerate}
\item Each standard module $M^n(\mu)_\pi$ has  a \\ unique simple quotient
$L^n(\mu)_\pi$.   Every simple representation of $\h_n$ is isomorphic to some $L^n(\mu)_{\pi}$.
\item $L^n(\mu)_\pi$ contains as a $S_n$-submodule the irreducible
$S_n$-representation corresponding to the partition $\pi$. 
\end{enumerate}
\end{theorem}
Zelevinsky proved this in \cite{Ze} in the language of $p$-adic groups.  The statement above is from Theorems 2.3.1 and 2.3.3 of \cite{Su}.
We will  interpolate the standard modules to complex rank.

\subsection{Interpolation of $\ev_a$} The evaluation homomorphisms themselves do not make sense in complex rank, but the next result shows that the associated functors do.

\begin{proposition} \label{evinter}  There is a functor over $\A_{\C}^1$, $\ev\colon B_{\A_{\C}^1}^{\A_{\C}^2}( \del) \to \delh$ inducing by fibers for $a \in \C$, functors $\ev_a\colon \del \to \delh$ that interpolate the   evaluation functors defined above.
\end{proposition}

For finite sets $U_1, \dots, U_m$, the space $\bigotimes_i \injr(U_i,I)$ has a $\C[S_n]$-basis indexed naturally by recollements $C$ of $U_1,
\dots, U_m$ such that $\abs{C} \leq n$.
This  is a generalization of Proposition~2.8 of \cite{De} and is proved precisely the same way. We shall use this in the proof.

\begin{proof} In the sequel, we shall abuse notation and occasionally blur the distinction between $[U]$ and $\injr(U,I)$. Consider the integral case.  Let $I= \{1, 2, \dots, n\}$.  
From the $\ev_a$ functor, we have a $S_n$-morphism  
\( [\one] \otimes [U] \to [U]\), which corresponds to an element $e \in [\one] \otimes [U] \otimes [U]$.
Let $y_i, i \in I$ denote the canonical basis for $[\one]$. 
Then  \[ {e = \sum_{i \in I, f \in \inj(U,I)} y_i \otimes f \otimes \left( af +  \frac{1}{2} \sum_{j \neq i}
(i \ j) \circ f \right) .} \]

 We compute the coefficients of  $e$ in terms of the canonical basis 
\( y_i \otimes f \otimes g,   i \in I, \  f,g \in \inj(U,I).\)
 Fix $i, f, g$.  Then  $y_i \otimes f \otimes g$
will occur in $e$ 
precisely if $f=g$ or if $f,g$  differ by a transposition starting at $i$, which corresponds to the image of $\one$ in the recollement induced by $i, f,g$.  

If $f \neq g$, then the  coefficient of $y_i \otimes f \otimes g$ is $1$.  If $f=g$, and $i \notin \im(f)$, the coefficient is  $a + \card I - \card U-1$; if $f=g$ and $i \in \im(f)$, the coefficient is $a$.  Otherwise, it is zero.

Thus, translating the basis elements into recollements as in Proposition~\ref{interpolateJM}, we can describe $[\one] \otimes [U] \to [U]$ by the  sum $\sum_{D \in \rec(\one, U, U)} Q^D(n,a) (D)$ (which by itself is a map $[\emptyset] \to [\one] \otimes [U] \otimes [U]$, but one can apply self-duality), where the $Q^D$ are polynomials. 
In general, in $\rep'(S_T, \C[T])$, the morphism $\f{h} \otimes M
\to M$ for $M=[U]$ is defined through $\sum_{D} Q_D(T,a) (D)$.  
We can check that the dAHA identity \eqref{eq: dAHA relation} holds by specializing $T \to n$ for $n$ large, and similarly for functoriality. The functors then extend to the pseudo-abelian envelope, and one then applies base extension to $\spec \mathbb{C}[T]$.
\end{proof}

\subsection{Interpolation of induction functors}

We now  construct the induction functors necessary to interpolate the standard modules.

\begin{lemma}  \label{injlemma} Let $I,J$ be finite sets, and let $S_I, S_J$ be the
corresponding permutation groups.  Then \( \indord_{S_I \times S_J}^{S_{I \sqcup J}} \left(
  \injr(V,I) \otimes \injr(J,J)\right) \simeq \injr(V \sqcup J, I \sqcup J),\) where $\sqcup$ denotes the disjoint union of sets.
\end{lemma}
\begin{proof} 
We have a $S_I \times S_J$-morphism
\( \injr(V,I) \otimes \injr(J,J) \to \resord^{S_{I \sqcup J}}_{S_I \times S_J} \injr(V \sqcup J, I \sqcup J) \)
that sends $e_f \otimes e_{\sigma}$ for $f\colon V \to I$ and $\sigma\colon J \to J$ to $e_{f \sqcup \sigma}$.  By
Frobenius reciprocity,  we get a $S_{I \sqcup J}$-map
\( \indord_{S_I \times S_J}^{S_{I\sqcup J}} \  \left( \injr(V,I)  \otimes \injr(J,J) \right) \to  \injr(V \sqcup J, I\sqcup J), \)
which is seen to be surjective as follows.   When $\abs{V} < \abs{I}$, the image contains at least
one basis element $e_h$ for  some  $h\colon V \sqcup J \to I \sqcup J$    
gluing together some  $f\colon V \to I, \id\colon J \to J$, and the  span of
the corresponding $S_{I \sqcup J}$ orbit is seen to be
the full space.  Since the dimensions of the two vector spaces are easily checked to be equal, 
 the map is bijective.  
\end{proof}

\begin{proposition} \label{inductionst}  Fix $N,n$.  Then for   all  $t$, there is a functor
  \[\ind\colon  \rep(S_t)^{(N)} \times \repord(S_n) \to \rep(S_{t+n})^{(N+n)}\]
such that   if $A \in \rep(S_T, \C(T)), B \in \repord(S_n)$, then for  almost all $t$, we have \(\ind  (\sp_t(A),B) \simeq \sp_{t}( \ind(A,B)) \in \rep(S_{t+n})\)
and for almost all $m \in \mathbb{Z}_{\geq 0}$,
\[\indord_{S_m \times S_n}^{S_{m+n}} (\ssp_m(A),B) \simeq \ssp_{m}( \ind(A,B)) \in \repord(S_{m+n}).\]
\end{proposition}

This could have been stated more closely in line with the framework developed in the previous section, but we thought it would only obscure the meaning.

\begin{proof} We first describe the functoriality of induction via recollements. As usual, we start with the integral case.  Assume $\abs{J}=n$ and $\abs{I}=m$. Given   $C \in \rec(V,U)$  and $D \in \rec(J, J)$ with $\abs{D} =  n$,  then
\[ \indord_{S_I \times S_J}^{S_{I\sqcup J}}( (C), (D)) = (C \sqcup D)\colon \injr(V \sqcup J,I \sqcup J) \to \injr(U \sqcup J,I \sqcup J),\] i.e., $(C
\sqcup D) \in \rec(V \sqcup J, U \sqcup J)$ is induced by $(C), (D)$.  Note that such $(C),(D)$ form a basis for the relative $\hom$-spaces if $m$ is large enough (and always span them).
In view of this, we now construct the interpolated $\ind$ on the full subcategory of $\rep(S_t) \times \repord(S_n)$ consisting of objects of
the form $([V],  [J])$, by sending $([V] , [J]) \to [V \sqcup J]$ and morphisms
$(C),(D)$  to $(C \sqcup D)$; functoriality  is a polynomial condition in terms of recollements and holds everywhere by testing at large integers.   The functor extends by additivity to the Karoubi envelope of this full subcategory, which up to equivalence contains as a full subcategory $\rep(S_t) \times \repord(S_n)$.  The specialization assertion follows from Lemma~\ref{injlemma}. 
\end{proof}

\begin{remark} {It does not seem
 possible to make more than one parameter non-integral in the above
 functor, because induction from $S_m \times S_n \to S_{m+n}$ multiplies
 the dimension by $\frac{(m+n)!}{m!n!}$, which is not a polynomial in
 $n,m$; if, however, one variable is fixed, then the expression becomes
 a polynomial in the other.}  Nevertheless, it is  possible to define restriction functors $\rep(S_{t+u}) \to \rep(S_t) \boxtimes \rep(S_u)$, by sending $\f{h} \to \f{h} \otimes \unital \oplus \unital \otimes \f{h}$ and extending tensorially. Consequently one may define  consequently a left adjoint (at least generically, when the categories are semisimple) from $\rep(S_t) \boxtimes \rep(S_u)$ into the ind-completion of $\rep(S_{t+u})$.  When one parameter is in $\mathbb{Z}_{\geq 0}$ and we use the ordinary category, we do not need ind-objects. 
 \end{remark}

In the next result, we do not obtain an everywhere defined induction functor.  It is an open question to determine ``degeneracy phenomena'' where it cannot be defined.

\begin{proposition} \label{inductionht} Fix $n,N \in \e{N}$.  Then for transcendental $t$, there is a functor
  \[ \ind\colon  \rep(\h_{t})^{(N)} \times \repord(\h_n) \to \rep(\h_{t+n})^{(N+n)} \] such that $A \in \rep(\h_T), B \in \repord(\h_n)$  imply that for almost  all $m \in \mathbb{Z}_{\geq 0}$,
 \(\ind_{\h_m \otimes \h_n}^{\h_{m+n}}(\ssp_m(A), B) \simeq \ssp_{m}(\ind(A, B)) \ \in \repord(\h_{m+n}).\)
\end{proposition}
Here $T$ is an indeterminate, and the category $\rep(\h_T, \C(T))$ is defined in the analogous way.
\begin{proof}
The   Poincar{\'e}-Birkhoff-Witt  basis theorem for the dAHA (cf. \cite{Kl}) implies that for   $ A \in \repord(\h_m)$, $B \in \repord(\h_n)$, we have \[ \res^{\h_{m+n}}_{S_{m+n}}
\ind_{\h_m \otimes \h_n}^{\h_{m+n}}( A,B) \simeq 
\ind_{S_m \times S_n}^{S_{m+n}}( \res^{\h_m}_{S_m}(A), \res^{\h_n}_{S_n}(B)).\]
By Proposition~\ref{inductionst}, we need only   describe the $y$-morphism.

First, we describe the functors appropriately in the integral case, as usual. Let   $I=\{1, 2, \dots, m\}, J=\{m+1, \dots, m+n\}$. 
Consider $A \in \repord(\h_m), B \in \repord(\h_n)$.  By adding a  direct summand with dAHA-action coming from $\ev$ if
necessary, we may assume that as a $S_m$ (resp. $S_n$) module,   $A = \bigoplus_{s \in S} [U_s] \in \repord(\h_m), B = \bigoplus_{s' \in S'} [J] \in \repord(\h_n)$  for some finite sets $U_s, S, S'$.    Now
\( \ind_{S_m \times S_n}^{S_{m+n}}\left( \bigoplus_{S} [U_s], \bigoplus_{S'} [J] \right) = \bigoplus_{S \times S'} [U_s+J]. \)

For simplicity, we assume $S$ and $S'$ consist of one element; the general case is handled similarly.  
The $y$-map $[U] \otimes [\one] \to [U]$ is described by some element \( \sum_{f, i, g} C_{f,i}^g e_f \otimes e_i \otimes e_g \in [U] \otimes [\one] \otimes [U].\)
Similarly, the map $[J] \otimes [\one] \to [J]$ is described by 
\(  \sum_{h, j, l} D_{h,j}^l e_h \otimes e_j \otimes e_l,\)
and the map $[U + J] \otimes [\one] \to [U+J]$ can be
described by some
\(  \sum_{p, k, q} E^q_{p,k} e_p \otimes e_k \otimes e_q.\)
Given $C_{f,i}^g, D_{h,j}^l$, we need to compute $E^q_{p,k}$.    Given the action of $y_i$ $(i 
\in I)$ on
$[U]$, to get an action of $y_i \ (i \in I)$ on $[U+J] $ we must add
$-\frac{1}{2}$ times the action of the cycles $(i \ j')$ for $j' \in
[m+1, n+m]$.  We similarly get the action of $y_{j}$ ($j \in J$) on $[U+J]$  using
the action of $y_j$ ($j \in J$) on $[U]$ and additional 2-cycles, by the definition of the homomorphism $\h_m \otimes \h_n \to \h_{m+n}$, which preserves the $x_i$ generators. 

Given $k \in I+J$ corresponding to the image of $\one$, $p \in \inj(U+J, I+J), q \in \inj(U+J, I+J)$, we compute $E^q_{p,k}$.  By $S_{I+J}$-invariance, assume $p$ comes from pasting $p'\colon U \to I$, $p'' = \id_J\colon J \to J$.  Suppose first $k \notin \im(p'') = J \subset I+J$.  Then  
 $E^q_{p,k}  = C^{q|_U}_{p',k}$ if $q|_J = \id_{J}$.  If $q|_J$  varies from the identity by a transposition starting at  $k$, then $E^q_{p,k} = -\frac{1}{2}$, and otherwise $E^q_{p,k}=0$.  
Now suppose $k \in \im(p'')=J$.  Then 
$E^q_{p,k} = D^{q|J}_{\id_J, k}$ if $q |_U = p|_U$, $E^q_{p,k}  = \frac{1}{2}$ if $q |_U, p|_U$ differ by a transposition starting at $k$ and $q|_J = \id_J$, and $E^q_{p,k}=0$ otherwise.  
 Thus, we have described the induced $y$-morphism in terms of recollements, and  interpolation can proceed.   

Finally, the dAHA identity \eqref{eq: dAHA relation} on each factor describes closed subschemes $V_1, V_2$ of some coproduct of affine spaces over $\C[T]$, and induction is a morphism $f$ into some coproduct of affine spaces, which has a closed subscheme $V_3$ from \eqref{eq: dAHA relation} too.  
Now, $(f(V_1 \times V_2))_r \subset (V_3)_r$ for $r \in \mathbb{Z}_{\geq 0}$ large, so we can apply the usual constructibility argument to see that the induced $y$-morphism satisfies \eqref{eq: dAHA relation} for all but possibly finitely many algebraic $t$. \end{proof}
This functor can  be iterated to multiple factors, a fact that we shall use below.  We shall keep the same notation.



\subsection{Simple objects in $\rep(\h_t)$}

\subsubsection{Interpolated standard modules}
Fix $N$.  Let $\pi\colon l_2 + \dots + l_k$ be an unordered partition of    $l \leq N$.  Then for $\bd{a} =(a_1, \dots, a_k) \in \t_k^*$, set \[M^t(\bd{a})_{\pi} = \ind^t_{t-l, l_2, \dots, l_k}( \ev_{a_1 + \rho_1}\unital, \dots, \ev_{a_k + \rho_k}\unital) \in \rep(\h_t ),\] where $\unital = [\emptyset]$.  This initially appears to be defined for $t$ transcendental, but in general we may define
\[ M^t(\bd{a})_{\pi} = \sp_t\left( \ind^T_{T-l, l_2, \dots, l_k}( \ev_{a_1 + \rho_1}\unital, \dots, \ev_{a_k + \rho_k}\unital)\right) \in \rep(\h_t ) .\] 
 By the construction of   Proposition~\ref{inductionht}, there is a morphism $s\colon \A_{\C}^1 \times_{\C} \A_{\C}^k \to \ob$, where  $\ob$ is the object scheme of $\delh^{(M, N)}$ for $M,N$ large, such that $s(t,a_1, \dots, a_k) = M^t(a_1, \dots, a_k)_{\pi}$.  

\newcommand{\ogl}{\mathcal{O}(\mathfrak{gl}_k)}

\subsubsection{Classification}

\begin{theorem} \label{bigthmyeah}
\begin{enumerate}
\item  For all transcendental $t$, every simple object in $\rep(\h_t)$ is a quotient of some $M^t(\bd{a})_\pi$.
\item Given an unordered partition $\pi$ and $t$ transcendental, the  
 set of $\bd{a} \in \mathbb{C}^k$ with $M^t(\bd{a})_\pi$ not simple is contained in a finite union of hypersurfaces. 
\end{enumerate}
\end{theorem}

\begin{proof} We shall prove this in steps.  The first part will follow from the next lemma.  
\begin{lemma}
 If $N \in \mathbb{Z}_{\geq 0}$, then outside a finite set $S_N \subset \C$ of algebraic numbers, any simple object in $\rep(\h_t, \C)^{(N,N)}$ is a quotient of some $M^t(\bd{a})_{\pi}$ with $\abs{\pi} \leq N$.
\end{lemma}

\begin{proof} It is enough to prove this for each $N$. Let $Sim$  be the constructible subset of the object scheme $\ob$ of $\delh^{(N,N)}$ as in Proposition~\ref{simplicityisgeneric},  and let $Ind \subset \ob  $ be the constructible set   parametrizing the images of $M^t(\bd{a})_{\pi}$ for $\abs{\pi} \leq N$, $\bd{a} \in \t_k^*, t \in \A_{\C}^1$.  Using  Proposition~\ref{injisgeneric}, we see that there is a constructible $Q \subset \ob$ with $Q_{\b{s}}$ consisting of quotients of objects  in $Ind_{\b{s}}$.  Now $Q_{\b{s}} \supset 
Sim_{\b{s}}$ for $s \in \mathbb{Z}_{\geq 0}$ sufficiently large by parts (1) and (2) of Theorem~\ref{zv}, so the statement holds generically.  Thus it holds at at least one transcendental $t$, so in view of the action of the Galois group $\mathrm{Gal}(\mathbb{C}/\mathbb{Q})$, at all transcendental $t$.
\end{proof}

For the next part, we prove:
\begin{lemma} \label{istilldonthavemystslaptop} Let $\pi$ be an unordered partition of $N < n$.  Then the set of $\bd{a} \in \C^k$ with $M^n(\bd{a})_{\pi} \in \repord(\h_n)$ not simple is contained in a finite union of hyperplanes.
\end{lemma}

First, we review  some known results.  Fix $n$.  Given $\lambda \in \t_k^*$, there is an \emph{Arakawa-Suzuki functor}, constructed in \cite{AS} for $\mathcal{O}(\mathfrak{sl}_k)$ and described  slightly more conveniently for us in \cite{Su}, \(F_{\lambda}\colon \ogl \to \repord(\h_n),\)
 where $\ogl$ is the \emph{BGG category} associated to the Lie algebra $\mathfrak{gl}_k$. {This is defined similarly to the category $\mathcal{O}$ associated to a semisimple Lie algebra, cf., e.g., \cite{Ga}.}  In detail, let $\frak{t}_k, \frak{n}_k$ be the subalgebras of upper-triangular and diagonal matrices, respectively.  An object $X \in \ogl$ is a representation of $\mathfrak{gl}_k$ which is
finitely generated over the enveloping algebra of $\frak{gl}_k$, semisimple as a representation of $\frak{t}_k$, and  locally finite under the action of $\frak{n}_k$. 

    For $\lambda \in  \t_k^*$, let $V(\lambda) \in \ogl$ be the Verma module   and $W(\lambda)$ its unique simple quotient.  If $\lambda = \mu + (n-l)\epsilon_1 + l_2 \epsilon_2 + \dots + l_k \epsilon_k$ is dominant, then by Theorems 3.2.1 and 3.2.2 of \cite{Su} we have \( F_{\lambda}(V(\mu))
= M^n(\mu)_\pi \) and \( F_{\lambda}(W(\mu)) = L^n(\mu)_{\pi} \  \text{or} \ 0 .\)
Recall that $L^n(\mu)_\pi$ is the unique simple quotient, as in Theorem~\ref{zv}.

\begin{proof}[Proof of Lemma~\ref{istilldonthavemystslaptop}] 
Consider the set of $\bd{a}$ such that $\bd{a} + (n-l, l_2, \dots, l_k)$ is dominant and $\bd{a}$ is antidominant.  The set of such $\bd{a}$ is a countable union of hyperplanes in $\A_{\C}^k  = \t_k^*$.  For such $\bd{a}$, the Verma module $V( \bd{a} )$ is irreducible (cf. \cite{Ga}) and \( W( \bd{a}  ) \simeq V( \bd{a}  ),\)
so it follows that \[F_{\bd{a} + (n-l, l_2, \dots, l_k)}(W(\bd{a}  )) \simeq
F_{\bd{a} + (n-l, l_2, \dots, l_k)}(V(\bd{a}  )) \simeq
 M^n(\bd{a} )_{\pi}   \neq 0.\]  Hence $ M^n(\bd{a})_{\pi} \simeq L^n(\bd{a})_{\pi}$, which is simple.  Thus the set of such $\bd{a}$ is contained in a \emph{countable} union of hyperplanes.  However, in view of Proposition~\ref{simplicityisgeneric}, the set of such $\bd{a}$ is constructible and thus contained in a finite such union.
\end{proof}

We now finish the proof of Theorem~\ref{bigthmyeah}.  Given $\pi$, there is a constructible set $Const \subset \A_{\C}^1 \times_{\C} \A_{\C}^k  $ such that the fiber over $t \in \mathbb{C}$ corresponds to $\bd{a} \in \C^k$ with $M^t(\bd{a})_\pi \in \rep(\h_t)^{(N,N)}$ simple; this is a consequence of Proposition~\ref{simplicityisgeneric}.  Then $\dim(\overline{Const}_n) \leq k-1$ when $n$ is a large integer by the previous lemma.  The fibers of $\overline{Const}$ have constant dimension on an open dense subset of $\A^1_{\C}$, so it follows that $Const_{t}$ is contained in a subvariety of proper codimension for almost all $t$.  In particular, this must include one transcendental $t$, and we can use the same Galois group argument as before.
\end{proof}

\newtheorem{claim}{Claim}

In the above proof, the second part of  Theorem~\ref{zv} was crucial, because it gave a classification of the irreducibles in the truncated categories $\repord(\h_n)^{(N)}$ consisting of modules whose $S_n$-restrictions contain irreducibles associated to Young diagrams with at most $N$ squares below the first row.  

\section{Wreath product categories}

\subsection{Definition of $\rep(\Al_t)$}
Let $A$ be a unital associative algebra over an algebraically closed field $k$ of characteristic zero.  
 Let $r\colon \f{h} \to \f{h} \otimes \f{h}$ be the dual to the algebra bilinear form   $\rho\colon \f{h} \otimes \f{h} \to \f{h}$, where $\f{h}$ is the regular representation of $S_n$.  We shall abbreviate $S_n \ltimes A^{\otimes n}$ to $\Al_n$ for ease of notation.

Consider the definition of a module $M$ over $\Al_n$:  $M$ is  a representation of $S_n$, together with linear maps $y_i\colon A \otimes M \to M$ for $1 \leq i \leq n$ that send $a \otimes m$ to $(1 \otimes 1 \otimes \dots \otimes a \otimes \dots \otimes 1) m$, where the $a$ is in the $i$-th place.  If $a \in A$, write $y_{i,a} \in \hom_{\mathbb{C}}(M,M)$ for the map $m \to y_i( a \otimes m)$.

  Then by definition, $\sigma y_{i,a} \sigma^{-1} = y_{\sigma(i),a}$ for $\sigma \in S_n$ and $y_{i,a} \circ y_{i,b} = y_{i,ab}$.  Also, $y_{i, 1} = \id_M$.  Finally,  
\( [y_{i,a}, y_{j,b}] =  0  \ \text{if } i \neq j \) and  \( [y_{i,a}, y_{i,b}]  = y_{i,[a,b]}  .\)
Conversely,  a system of maps $y_{i,a}$ for $1 \leq i \leq n, \ a \in A$ satisfying the above  conditions on a $S_n$-representation $M$ yields an action of $\Al_n$.  The next result now follows.

\begin{proposition} An $\Al_n$-module $M$ is   an $S_n$-representation $M$ together with maps $y_a\colon \f{h} \otimes M \to M$ for $a \in A$ such that the association $a \to y_a$ is $k$-linear, $y_1 = s \otimes \id_M$ for $s\colon \f{h} \to \unital$ the standard morphism, and such that $y$ satisfies the following two further conditions:
\begin{equation} \label{eq:firstcondition}  y_a \circ (\id_{\f{h}} \otimes y_b) \circ r = y_{ab}\colon \f{h} \otimes M \to M \end{equation}
\begin{equation}  \label{eq:secondcondition}
y_a \circ  (\id_{\f{h}} \otimes y_b) - y_b \circ (\id_{\f{h}} \otimes y_a) \circ P_{1,2}  = y_{[a,b]} \circ (\rho \otimes \id_{\f{h}} ) \colon \f{h} \otimes \f{h} \otimes M \to M.
\end{equation}
\end{proposition}

All this is in a category-friendly stage, and it is now clear how to make the following definition:

\begin{definition}[P. Etingof, personal communication]
The category $\rep(\Al_t)$ is defined as follows.  Objects in $\rep(\Al_t)$ are objects $M \in \rep(S_t)$ endowed with a morphism
$ y\colon A \otimes \f{h} \otimes M \to M, $
where $A$ is viewed as a vector space, 
such that the induced morphisms $y_a\colon \f{h} \otimes M \to M$ satisfy \eqref{eq:firstcondition} and \eqref{eq:secondcondition} above, and $y_1 = s \otimes \id_M$.
Morphisms in $\rep( \Al_t)$ are morphisms in $\rep(S_t)$ commuting with the $y$ maps.
\end{definition}

Again, it is enough to specify $y_a$ for $a$ generating $A$.

\begin{proposition} If $A$ is finitely generated, there is a category over $\A_{k}^1$, $\rrep(\Al)$, whose fiber at each $t \in k$ is $\rep(\Al_t)$ as defined above.  If $R, S \in \mathbb{Z}_{\geq 0}$, there are subcategories $\rrep(\Al)^{(R,S)}$ of finite type.
\end{proposition}
\begin{proof} The result is now evident from the remark preceding the proposition when $A$ is finitely \emph{presented}. However, when $M$ is of fixed size, e.g.,  is an object of $\rep(S_t)^{(R,S)}$, then the $y$-morphism is described by certain parameters in $k$ depending on recollements whose number is bounded in terms of $R$ and $S$.  The relations imposed on $y$ are simply polynomial relations in these parameters, and by the noetherian property of a polynomial ring, it is clear that we need only impose finitely many relations on $y$ (where finite depends on $R$ and $S$).
\end{proof}
\newcommand{\cchar}{\mathrm{Char}}

We now carry out the interpolation of certain  important structures. 

\subsection{$V_\chi$ functors} 
Fix $n \in \mathbb{Z}_{\geq 0}$.  
Let $A$ be a finitely generated (hence finitely presented) commutative algebra over $k$.  In this case, we can get the simple objects in $\repord(\Al_n)$ by a certain process of parabolic induction, which begins with an analog of the evaluation homomorphism for $\h_n$.
 
Let $\cchar(A)$ be the collection of characters of $A$, i.e., algebra-homomorphisms $A \to k$. For $\chi \in \cchar(A)$, we define the functor $V_\chi\colon \repord(S_n) \to \repord(\Al_n)$ as follows.  
Choose    $M \in \repord(S_n)$.  Define $V_{\chi}(M)$ to be $M$ as a $S_n$-module, and for $a \in A$, let $1 \otimes \dots \otimes a \otimes \dots \otimes 1 \in A^{\otimes n}$ act on $M$ via scalar multiplication by $\chi(a)$.   
If $f\colon M \to N$ is a $S_n$-homomorphism, then $V_{\chi}(f)\colon V_\chi(M) \to V_{\chi}(N)$ is set-theoretically the same function $f$.

 There is a scheme of finite type over $k$, which by abuse of notation we denote by $\mathcal{C}har(A)$, parametrizing the characters of $A$.

\begin{proposition} There is a functor over $\A_{k}^1$, $V\colon \mathcal{C}har(A) \times_k \del \to \rrep(\Al)$, that interpolates the previous functors $V_{\chi}$ at natural number points and $\chi \in \mathcal{C}har(A)$.
\end{proposition}

\begin{proof} We start by working in the category $\rep'(\Al_T, k[T])$ defined in an obvious way.  
Let $a_1, \dots, a_r$ be a generating set for $A$.  
Let $s\colon \f{h} \to \unital$ be the standard map.  Then if $M \in \rep'(S_T)$, define $y_{a_i}\colon \f{h} \otimes M \to M$ by $\chi(a_i) (s \otimes \id)$.  We claim that this defines a functor $\rep'(S_T) \to \rep'(\Al_T)$.
Indeed, the fact that \eqref{eq:firstcondition} and \eqref{eq:secondcondition} hold can be checked by specialization $T \to n$ for $n$ large, since it is easy to see that the following definition of $y_{a_i}$ is just a category-friendly version of the usual one for the $V_{\chi}$ functor.  It is similarly checked that for a morphism $(C)\colon [U] \to [V]$, the diagram
\[
\xymatrix{
\f{h} \otimes [U] \ar[d]^{\id_{\f{h}} \otimes (C)} \ar[r]^{y_{a_i}} & [U] \ar[d]^{(C)}\\
\f{h} \otimes [V] \ar[r]^{y_{a_i}} & [V] }
\]
 commutes for each $i$, so functoriality follows.   Then we can specialize the functor appropriately, $T \to t$, and extend to the pseudo-abelian envelope.
\end{proof}

\subsection{Induction functors}

We momentarily drop the hypothesis that $A$ be commutative. Now fix $n_2, n_3, \dots, n_l \in \mathbb{Z}_{ \geq 0}$.  Let $N = \sum n_s$.
\begin{proposition} \label{inductionsemidirectat} Fix a cutoff $R$. For transcendental $t$, there is a functor \[\ind\colon \rep(\Al_{t-N})^{(R)} \times \prod_{s=2}^l
\repord(\Al_{n_s})
 \to \rep(\Al_t)^{(R+N)}\] interpolating the usual functors.
\end{proposition}

\begin{proof} We handle the case $l=2$, because the general case can be obtained by iteration.  Fix $a\in A$. Let $t \in \mathbb{Z}_{\geq 0}$.  Assume for simplicity that the two objects restrict as $S_t$-objects to $[U] \in \repord(S_{t-N}), [J] \in \repord(S_N)$, where $\abs{J} = N$.  Then there are maps
$y\colon [\one] \otimes [U] \to [U], y'\colon [\one] \otimes [J] \to [J]$ that describe the action of $a$; we have dropped the subscript for convenience.  Let 
$y$ be described by $\sum_{f,i,g} C^{g}_{f,i} e_f \otimes e_i \otimes e_g \in [U] \otimes [\one] \otimes [U]$. 
Similarly, the map $y'\colon [J] \otimes [\one] \to [J]$ is described by 
\(  \sum_{h, j, d} D_{h,j}^{l} e_h \otimes e_j \otimes e_d.\) 
Finally, the map $[U + J] \otimes [\one] \to [U+J]$ can be
described by some
\(  \sum_{p, m, q} E^{q}_{p,m} e_p \otimes e_m \otimes e_q.\)
Given $C_{f,i}^{g}, D_{h,j}^{d}$, we need to compute $E^{q}_{p,m}$.

Given $m \in I+J$ corresponding to the image of $\one$, $p \in \inj(U+J, I+J), q \in \inj(U+J, I+J)$, we compute $E^{q}_{p,m}$.  By $S_{I+J}$-invariance, assume $p$ comes from pasting $p'\colon U \to I$, $p'' = \id_J\colon J \to J$.  Suppose first $k \notin \im(p'') = J \subset I+J$.  Then  
 $E^{q}_{p,m}  = C^{q|_U}_{p',m}$ if $q|_J = \id_{J}$.   Otherwise $E^q_{p,m}=0$.

  We do a similar procedure when $m \in \im(p'')$.  Then,
$E^q_{p,m} = D^{q|_J}_{\id_J, m}$ if $q(J) \subset J$ and $p|_U = q|_U$, and $E^q_{p,m}=0$ otherwise.

Using this, we can express the functor in terms of recollements (recall that $E^q_{p,m}$ becomes the coefficient for the recollement defined by $p,q,m$), and argue as in the proof of induction for the dAHA (Proposition~\ref{inductionht}).
\end{proof}

\subsection{Some combinatorics} We start by describing the construction of simple objects in the standard case.  Fix $n \in \mathbb{Z}_{\geq 0}$.  
\begin{theorem}[Classical] \label{classicalsimplesemidirectprod} Suppose $A$ is commutative and finitely generated. Every simple object in $\repord(\Al_n)$ can be obtained from a partition $n = n_1 + \dots + n_l$, characters $\chi_1, \dots, \chi_l$, objects $M_j \in \repord(S_{n_j})$ for $1 \leq j \leq l$, via \\
\[  \ind_{\Al_{n_1}  \otimes \dots \otimes  \Al_{n_l}}^{\Al_n}( V_{\chi_1}(M_1), \dots,  V_{\chi_l}(M_l) ), \]
where the $ M_j \in \rep(S_{n_j})$ are simple, and $  \chi_j \in \cchar(A)$. 
\end{theorem}

For a proof when $A$ is  the group algebra of a finite group, cf. \cite{EtVaetal}.  We shall derive a similar classification in transcendental rank.
 In order to apply the machinery developed, we will need corresponding results on the filtration of these categories.

\begin{corollary} \label{serrelocalfieldsisawesome} Suppose $n>2N$.  Then every simple object in $\repord(\Al_n)^{(N)}$ is of the form  $\ind( V_{\chi_1}(M_1), \dots,  V_{\chi_l}(M_l) )$ for $n_2 + \dots + n_l \leq N$ and $M_1 \in \repord(S_{n_1})^{(N)}$ (all $M_j$ simple).
\end{corollary}

The corollary is a consequence of Theorem~\ref{classicalsimplesemidirectprod} and the next lemma.
 
\begin{lemma} \label{keyboundkey} Let $n>2N$.  Suppose $\mu\colon n_1 + n_2 + \dots + n_l =  n $ is an ordered partition of $n$, $V_j \in \repord(S_j)$ is nonzero, and 
\( \ind_{S_{n_1}   \times \dots \times S_{n_l}}^{S_n}(\bigotimes V_j) \in \repord(S_n)^{(N)} .\)
Then $n_2 + \dots + n_l \leq N$ and $V_1 \in \repord(S_{n_1})^{(N)}$.
\end{lemma}

\begin{proof}  Without loss of generality, we may assume each $V_j$ simple.
Recall what it means for an irreducible object in $\repord(S_n)$, parametrized by a Young diagram, 
to belong to $\repord(S_n)^{(N)}$.  It means that below the first row, there are at most $N$ squares.  In particular, if $\lambda, \mu$ are Young diagrams and $\lambda > \mu$ in the lexicographical order, then $\lambda \in \repord(S_n)^{(N)}$ if $\mu \in \repord(S_n)^{(N)}$.  

Now, by the following lemma, the lowest simple object in the $S_n$-representation 
$\ind_{S_{n_1} \times S_{n_2} \times \dots \times S_{n_l}}^{S_n}( V_1 \otimes \dots \otimes V_l)$, with respect to the lexicographical order, has a top row in the associated Young diagram of length at most $n_1$.  In particular, the rows below it  have total length at least $n_2 + \dots + n_l$, so this is at most $N$.  Using similar reasoning, the number of boxes in the Young diagram of $V_1$ below the first row can be at most $N$. 
\end{proof}

\begin{lemma} Let $n = n_1 + \dots + n_l$, let $V_{\lambda_i}  \in \repord(S_{n_i})$ be the irreducible object associated to the partition $\lambda_i$ of $n_i$, and let $V = \ind_{S_{n_1} \times \dots \times S_{n_l}}^{S_n} ( \bigotimes_i V^i_{\lambda_i})$.
Then for some $C_\mu$,
\( V = V_{\lambda} \oplus \bigoplus_{\mu\colon \mu > \lambda} C_{\mu} V_{\mu } \)
where $\lambda$ is the partition of $n$ obtained by splicing together all the $\lambda_i$ (and possibly reordering) and $V_{\lambda} \in \repord(S_n)$ the associated simple representation.
\end{lemma}

This is essentially the branching rule for the symmetric group.

\subsection{Generic classification of simple objects} 
Using the machinery developed, we can now extend this result to complex rank generically:

\begin{theorem} For $t$ transcendental, every simple object in $\rep(\Al_t)$ is of the form 
\( \ind( V_{\chi_1}(M_1), V_{\chi_2}(M_2), \dots,  V_{\chi_l}(M_l) ) \)
where
$M_1 \in \rep(S_{t-N})$ is simple, $ M_j \in \repord(S_{n_j})$ is simple for $j >1$, $n_2 \geq n_3 \geq \dots$,   $N = \sum_2^l n_j$, and each $\chi_j \in \cchar(A)$.
Moreover, different choices yield non-isomorphic objects.
\end{theorem}

In other words, the simple objects are indexed by families of Young diagrams of arbitrary size, together with the assignment of a character to each diagram.  To the top one, one can append a ``very long line'' as usual.

The result follows from the next proposition, 
since  the difference between objects associated to different Young diagrams or character arrangements is implied by  the classification at the integers.

\begin{proposition} \label{cutoffprop}
If $W$ is fixed, then for all but finitely many $t$ the simple objects in 
$\rep(\Al_t)^{(W,W)}$
are of the form 
\( \ind( V_{\chi}(M_1), V_{\chi_2}(M_2), \dots,  V_{\chi_k}(M_k) ) \)
with $M_1 \in \rep(S_{t-N})^{(N)}, M_j \in \repord(S_{n_j})$ simple, and $N \leq W$.
\end{proposition}
The object $\ind( V_{\chi}(M_1), V_{\chi_2}(M_2), \dots,  V_{\chi_k}(M_k) )$ may be defined by choosing $M_1' \in \rep(\Al_{T-N}, \C(T))$ specializing to $M_1$, applying the functor of Proposition~\ref{inductionsemidirectat}, and specializing.

\begin{proof} 
The images of such objects form a constructible subset $T$ of $\ob$, the object scheme of $\rrep(\Al)^{(W,W)}$, in view of Proposition~\ref{simpleisgeneric}.  Applying Proposition~\ref{simpleisgeneric} again to $\rrep(\Al)^{(W,W)}$, we see that  $T_{\bar{t}} \supset Sim_{\bar{t}}$ for all but finitely many $t$, since this is true for the (Zariski dense) subset of sufficiently large integers by Corollary~\ref{serrelocalfieldsisawesome}.
\end{proof}

\subsection{The case of $A$ noncommutative}

We shall briefly indicate the changes necessary when $A$ is noncommutative.  For $n \in \mathbb{Z}_{\geq 0}$, it is known that the finite-dimensional representations of $\Al_n$ can be obtained via a similar process of parabolic induction, but where characters $\chi$ are replaced with irreducible representations of $A$. 
If $\chi$ is an irreducible representation of $A$, define the functor $V_{\chi}\colon \repord(S_n) \to \repord(\Al_n)$ by sending $M \in \repord(S_n)$ to $M \otimes \chi^{\otimes n}$, where $S_n$ permutes the factors of $\chi^{\otimes n}$.
Every finite-dimensional irreducible $\Al_n$ representation can be obtained from a partition $n = n_1 + \dots + n_l$, irreducible representations $\chi_1, \dots, \chi_l$ of $A$, irreducible $S_{n_j}$ representations $M_j$ for $1 \leq j \leq l$, via the familiar formula: 
$\ind_{\Al_{n_1} \otimes \dots \otimes \Al_{n_l}}^{\Al_n}(V_{\chi_1}(M_1), \dots, V_{\chi_l}(M_l))$.

Now take $n  \gg N$.
It follows that the simple $\Al_n$ representations that restrict to the trucation $\repord(S_n)^{(N,N)}$ are obtained via such a procedure with $n_2  + \dots + n_l \leq N$, $M_1 \in \repord(S_{n_1})^{(N)}$ simple, the dimensions of the $\chi_j, 1 \leq j \leq l$ bounded by some bound depending on $N$, and $\chi_1$ a one-dimensional representation (i.e., a character).  Indeed, the bound on $n_2 + \dots + n_l$ and the fact that $M_1 \in \repord(S_{n_1})^{(N)}$ follow from Lemma~\ref{keyboundkey}.  It follows that $\chi_1$ is one-dimensional, or otherwise $V_{\chi_1}(M_1)$ would be too large because of the tensor power and the induced object would not belong to $\repord(S_n)^{(N,N)}$.  Similarly, the dimensions of the $\chi_j, 2 \leq j \leq l$, must be bounded in terms of $N$ for $n$ sufficiently large.

Thus, in the same way, a version of Proposition~\ref{cutoffprop} holds for $A$ noncommmutative (but with $\chi_1$ required to be one-dimensional), because the $\chi_j$ can all be parametrized by suitable schemes, since their dimensions are bounded depending only on $N$.
It follows that:
\begin{theorem} \label{noncommutative} For $t$ transcendental, the simple objects in $\rep(\Al_t)$ for $A$ finitely generated are obtained via
$\ind(V_{\chi_1}(M_1), \dots, V_{\chi_l}(M_l))$ for each $M_j$ simple and $\chi_1$ one-dimensional.
\end{theorem}

Curiously, if $A$ has no one-dimensional representations (e.g., is a simple algebra), then $\rep(\Al_t)$ is empty for $t$ transcendental; this is an immediate corollary of Theorem~\ref{noncommutative}.
The author is indebted to Etingof for pointing out this generalization.

\subsection{Tensor structure   for $A$ a Hopf algebra}  We now  prove:
\begin{theorem} \label{hopfalgebraimpliestenscat} \label{tensorstructure} If $A$ is a Hopf algebra, 
then $\rep(\Al_t)$ can be  made into a tensor category, in a manner interpolating the usual tensor structure (using the Hopf algebra structure on $\Al_n$ for $n \in \mathbb{Z}_{\geq 0}$).  Indeed, the tensor structure becomes a bifunctor over $\A_k^1$.
\end{theorem}

For any object $M \in \rep(S_t)$, set \( Q_1(M) = \hom( \f{h} \otimes M, M), \) and \(   Q_2(M) = \hom( \f{h} \otimes \f{h} \otimes M, M).\)
Both these can be made into associative rings in view of the map $r\colon \f{h} \otimes \f{h} \to \f{h}$ and its dual $\rho\colon \f{h} \to \f{h} \otimes \f{h}$. 
When $t \in \mathbb{Z}_{\geq 0}$ is very large, we can think of $Q_1(M)$ as parametrizing families of maps $f_i\colon M \to M, 1 \leq i \leq n$, such that $\sigma f_i \sigma^{-1} = f_{\sigma(i)}$ for $\sigma \in S_n$.  The ring structure comes from composing these families of maps.  Similarly, $Q_2(M)$ parametrizes families of maps $f_{ij}\colon M \to M, 1 \leq i,j \leq n$ such that  $\sigma f_{ij} \sigma^{-1} = f_{\sigma(i)\sigma(j)}$.  To check that the ring is associative in general, one reduces to the case $M = [S]$ for $S$ a finite set.  Then associativity is a polynomial condition in $t$, and its holding at sufficiently large natural numbers implies it at all $t$.  Similarly, it can be checked that there are algebra-homomorphisms $Q_1(X) \otimes_k Q_1(Y) \to Q_1(X \otimes Y), Q_2(X) \otimes_k Q_2(Y) \to Q_2(X \otimes Y)$ for all $X,Y$.

Suppose given a $y$-map $A \to Q_1(M)$ for $M \in \rep(S_t)$.  In order that it define a structure of $ \rep(\Al_t)$, there are certain conditions it must satisfy.  We shall first introduce some notation.

Define a $k$-linear map $\theta \colon Q_1(M) \otimes_k Q_1(M) \to Q_2(M)$ as follows.  Given $f,g\colon \f{h} \otimes M \to M$, define $\theta(f \otimes g)\colon \f{h} \otimes \f{h} \otimes M \to M$ by
\[ \theta(f \otimes g) =  f \circ (\id_{\f{h}} \otimes g ) -   g \circ (\id_{\f{h}} \otimes f ) \circ P_{1,2}.\]
Define $\gamma, \gamma'\colon Q_1(M) \otimes_k Q_1(M) \to Q_2(M)$ by 
\[ \gamma(f \otimes g) = f \circ (\id_{\f{h}} \otimes g) , \quad \gamma'(f \otimes g) = g \circ (\id_{\f{h}} \otimes f)   \circ P_{1,2}   .\]
Note that $\theta = \gamma - \gamma'$.  In the integer case, $\theta$ sends $\{f_i\} \otimes \{g_j\}$ to $f_i g_j - g_j f_i$, and there are similar computations for $\gamma, \gamma'$.
Define $\phi\colon Q_1(M) \to Q_2(M)$ by setting $\phi(f)\colon \f{h} \otimes \f{h} \otimes M \to M$ to be
\( \phi(f) = f \circ (\rho \otimes \id_M).\)
It is easy to check that $\phi$ is a ring-homomorphism, by reduction to the integral case.

Finally, we may state the two conditions for when $y\colon A \to Q_1(M)$ defines on $M \in \rep(S_t)$ the structure of an object in $\rep(\Al_t)$.
First,  $y$ must be a (unital) ring-homomorphism.
Second, the following diagram must commute:
\begin{equation} \label{hardcondition}
\xymatrix{
A \otimes_k A \ar[d]^{[\cdot, \cdot]} \ar[r]^-{y \otimes y} & Q_1(U) \otimes_k Q_1(U) \ar[r]^-{\theta} &  Q_2(U) \\
A \ar[r] &    Q_1(U) \ar[ur]_-{\phi} } \end{equation}
Here, of course, $[\cdot, \cdot]\colon A \otimes_k A \to A$ is the commutator map.

As a result, the following is immediate:
\begin{proposition} \label{yeah} An algebra homomorphism $A \to B$ induces a functor $\rep(\mathbf{B}_t) \to \rep(\Al_t)$.
\end{proposition}

We shall now derive certain commutation relations in $\rep(S_t)$ necessary to the proof of Theorem~\ref{tensorstructure}.

\begin{lemma} The following diagram is commutative for all $X,Y$:\\
\begin{equation} \label{grothendieck} \xymatrixcolsep{7pc}
\xymatrix{
Q_1(X) \otimes_k Q_1(Y) \otimes_k Q_1(X) \otimes_k Q_1(Y) \ar[d] \ar[r]^-{\theta_X \otimes \gamma_Y + \gamma'_X \otimes \theta_Y } & Q_2( X ) \otimes_k Q_2(Y) 
\ar[d] \\
Q_1(X \otimes Y) \otimes_k Q_1(X \otimes Y) \ar[r]^-{\theta_{X \otimes Y}} &  Q_2(X \otimes Y) 
} \end{equation} 
\end{lemma}
\begin{proof} 
The result may be seen by reducing to the integral case, since a nonzero 1-variable polynomial has finitely many roots; this argument will recur.
Fix $X, Y \in \repord(S_n)$ and families $f_i, f'_j\colon X \to X, g_i, g'_j\colon Y \to Y$ corresponding to elements $f, f', g, g'$ of $Q_1(X), Q_1(Y)$.
Then going down and right in the diagram yields the family of maps
\begin{align*} h_{ij} & = (f_i \otimes g_i ) \circ (f'_j \otimes g'_j) - 
(f'_j \otimes g'_j) \circ (f_i \otimes g_i ) \\
& = ( f_i f'_j - f'_j f_i) \otimes  g_i g'_j - f'_j f_i \otimes (g'_j g_i - g_i g'_j) \\
& = \theta_X(f \otimes f') \otimes \gamma_Y(g \otimes g') + \gamma'_X(f \otimes f') \otimes \theta_Y(g \otimes g') . \end{align*}
\end{proof}

\begin{lemma} For any $X,Y$, the following diagram is commutative:
\begin{equation} \label{serre}
\xymatrix{
Q_1(X) \otimes_k Q_1(X) \otimes_k Q_1(Y) \ar[d]^{m_{Q_1(X)} \otimes \id} \ar[r]^-{\gamma_X \otimes \id} & Q_2(X) \otimes_k Q_1(Y) \ar[r]^-{\id \otimes \phi_Y} &  Q_2(X) \otimes_k Q_2(Y) \ar[d]  \\
Q_1(X) \otimes_k Q_1(Y) \ar[r] & Q_1(X \otimes Y) \ar[r]^{\phi_{X \otimes Y}} & Q_2(X \otimes Y) 
} \end{equation}
\end{lemma}
Here $m_{Q_1(X)}$ denotes the multiplication map.
\begin{proof} As usual, we need only consider the integral case.
Consider $f, f' \in Q_1(X), g \in Q_1(Y)$ parametrized by maps $\{f_i\}, \{f'_i\}, \{ g_i \}$.  Then going down and right yields the family of maps 
\( h_{ii} =   f_i f'_i  \otimes g_i  \) and $h_{ij}=0$ if $i \neq j$.  
Going one step to the right from the top-left square yields  $
\{ f_i f'_j   \} \otimes g \in Q_2(X) \otimes Q_1(Y)$, where the first $\{f_i f'_j   \}$ is intepreted as a family of maps (hence is in $Q_2(X)$).  Then, going right again and down yields
the same $h_{ij}$, as is easily checked. 
\end{proof}

\begin{proposition} \label{tensprodfirstcase} Let $C = A \otimes_k B$. There is a tensor product bifunctor
\( \rep(\Al_t) \times \rep(\mathbf{B}_t) \to \rep(\mathbf{C}_t) \)
sending $X,Y \to X \otimes Y$ such that if $A \to Q_1(X), y\colon B \to Q_1(Y)$ are the appropriate homomorphisms, then one takes $y\colon A \otimes_k B \to Q_1(X) \otimes_k Q_1(Y) \to Q_1( X \otimes Y)$.
\end{proposition}

\begin{proof} It is clear that $A \otimes_k B \to Q_1(X \otimes Y)$ is a ring-homomorphism.  We must show that the second condition is satisfied.  All homomorphisms $A \to Q_1(X), B \to Q_1(Y), A \otimes_k B \to Q_1(X \otimes Y)$ will be denoted by $y$; this abuse of notation should cause no confusion.
Fix $a \otimes b, a' \otimes b' \in A \otimes_k B$.
Then, since $[a \otimes b, a' \otimes b'] = [a,a'] \otimes bb' + a'a \otimes [b, b']$,
\begin{equation} \label{atiyah}
\phi(y_{[ a \otimes b, a' \otimes b']}) 
 = \phi( y_{[a,a']} \otimes y_{bb'} + y_{a'a} \otimes y_{[b,b']}) \end{equation}
Here we use $y_{[a,a']} \otimes y_{bb'}$ to refer to the image of it in $Q_1(X \otimes Y)$, by abuse of notation.

Next, by \eqref{grothendieck}, \eqref{serre} and its analog for $\gamma'$, and \eqref{atiyah},  
\begin{align*}
 \theta( y_{a \otimes b} \otimes y_{a' \otimes b'}) 
& = \theta_X(y_a \otimes y_{a'}) \otimes \gamma_Y( y_b \otimes y_{b'}) + \gamma'_X( y_{a} \otimes y_{a'}) \otimes \theta_Y( y_b  \otimes y_{b'}) \\
& = \phi_X(y_{[a,a']}) \otimes \gamma_Y( y_b \otimes y_{b'}) + \gamma'_X( y_{a} \otimes y_{a'}) \otimes \phi_Y( y_{[b,b']}) \\
& =  \phi_{X \otimes Y}( y_{[a,a']} \otimes y_{bb'} + y_{a'a} \otimes y_{[b,b']})
\\ & = \phi_{X \otimes Y}(y_{[ a \otimes b, a' \otimes b']}) = \phi( y_{[a \otimes b, a' \otimes b']} ).
\end{align*}
It follows that the homomorphism $y\colon A \otimes_k B \to Q_1(X \otimes Y)$ satisfies \eqref{hardcondition}.

Suppose given homomorphisms $f\colon X \to X', g\colon Y \to Y'$ where $X, Y \in \rep(\Al_t),$ $  X',Y' \in \rep(\mathbf{B}_t)$.  Then, we can consider $f \otimes g$ as a morphism in $\rep(S_t)$ from $X \otimes  Y \to X' \otimes Y'$.  It must, however, be checked that the commutation with the $y$-maps is preserved.  This is an easy calculation.
We have commutative diagrams:
\begin{equation} \label{twocd}
\xymatrix{
\f{h} \otimes X \ar[r]^{y_a} \ar[d]^{\id_{\f{h}} \otimes f} &  X \ar[d]^f \\
\f{h} \otimes X' \ar[r]^{y_a} \ar & X' }  \quad \quad  \xymatrix{ 
\f{h} \otimes Y \ar[r]^{y_b} \ar[d]^{\id_{\f{h}} \otimes g} &  Y \ar[d]^g \\
\f{h} \otimes Y \ar[r]^{y_b}  & Y } \end{equation}

Consider the diagram:
{\tiny
\[ 
{ \xymatrixcolsep{4pc}
\xymatrix{
\f{h} \otimes X \otimes Y \ar[r]^-{\rho \otimes \id} \ar[d]^{\id \otimes f \otimes \id} & \f{h} \otimes \f{h} \otimes X \otimes Y \ar[r]^{\id \otimes y_a \otimes \id} \ar[d]^{\id \otimes f \otimes \id}   & \f{h} \otimes  X \otimes Y  \ar[r]^-{P_{1,2} \circ (y_b \otimes \id) \circ P_{2,3}}  \ar[d]^{\id \otimes f \otimes \id} & X \otimes Y \ar[d]^{f \otimes \id}
\\
\f{h} \otimes X' \otimes Y \ar[d]^{\id \otimes \id \otimes g} \ar[r]^-{\rho \otimes \id} & \f{h} \otimes \f{h} \otimes X' \otimes Y  \ar[r]^{\id \otimes y_a \otimes \id} \ar[d]^{\id \otimes \id  \otimes g} & \f{h} \otimes X' \otimes Y \ar[r]^-{P_{1,2} \circ (y_b \otimes \id) \circ P_{2,3}} \ar[d]^{\id \otimes \id \otimes g}  & X' \otimes Y \ar[d]^{\id \otimes g} \\
\f{h} \otimes X' \otimes Y' \ar[r]^-{\rho \otimes \id} & \f{h} \otimes \f{h} \otimes X' \otimes Y'  \ar[r]^{\id \otimes y_a \otimes \id}  & \f{h} \otimes X' \otimes Y' \ar[r]^-{P_{1,2} \circ (y_b \otimes \id) \circ P_{2,3}} & X' \otimes Y'
}} \vspace{2mm}\] }
It is easy to see that each small square is commutative from \eqref{twocd}, so the whole square is commutative.  This implies that $f \otimes g$ commutes with the $y$-maps on $X \otimes Y$.   Thus, we have a functor indeed.
\end{proof}

\begin{proof}[Proof of Theorem~\ref{hopfalgebraimpliestenscat}]  
When $A$ is a Hopf algebra, the comultiplication homomorphism $\Delta \colon A \to A \otimes_k A$ yields a   functor  
\( \rep( ( \Al \otimes_k \Al)_t )  \to \rep(\Al_t) \)
by Proposition~\ref{yeah}, and consequently induces a monoidal structure on $ \rep(\Al_t)$ in view of Proposition~\ref{tensprodfirstcase}.
   The associativity property of this functor  follows from the associativity of the comultiplication,
\( (\Delta \otimes \id_A) \circ \Delta = (\id_A \otimes \Delta) \circ \Delta,\)
and the pentagon axiom follows since it is true in $\rep(S_t)$.  

Next, we construct the unital object. 
We take $\unital \in \rep(S_t)$ with the homomorphism $A \to Q_1(\unital)$ from 
$A \to k \to Q_1(\unital)$, where $A \to k$ is the counit and $k \to Q_1(\unital)$ the unit; this is an object of $\rep(\Al_t)$ because $Q_1(\unital), Q_2(\unital)$ are commutative.  The unital axioms now follow from those of the Hopf algebra, so that $\rep(\Al_t)$ is indeed a monoidal category.

We now show that $\rep(\Al_t)$ admits   duals,  which will complete the proof of the theorem.
First, for all $M \in \rep(S_t)$, the $k$-algebra $Q_1(M)$ has an anti-involution $\tau$ obtained as follows: if $M=[U]$, then the recollement $(C)$ of $\one, U, U$ is sent to the one of $\one, U, U$ reversing the order of the $U$'s.  This is to be thought of in the following manner.  If $|I|=t$, $\injr(U,I)$ has an $S_t$-invariant bilinear form.  Given a family $\{f_i\} \in Q_1(M)$, $\tau$ sends this to the family $\{ f_i^{\vee} \}\colon M \to M$.
There is a similar anti-involution on $Q_2(M)$.

Given $M_A \in \rep(\Al_t)$ now, we define the right dual $M_A^{\vee}$ as follows: as an object of $\rep(S_t)$, it is the same, which we will denote $M$ to emphasize that $M$ itself has no $A$-action.   If $y\colon A \to Q_1(M)$ is the morphism defining the $A$-structure on $M_A$, then consider the homomorphism
\( y'\colon A \xrightarrow{S}   A \to  Q_1(M)   \xrightarrow{\tau}   Q_1(M)  \)
as defining the $A$-structure on $M_A^{\vee}$, where $S\colon A \to A$ is the antipode map.
We now check that the diagram \eqref{eq:secondcondition} is commutative.  Note that $\theta( \tau f \otimes \tau g) = -\tau \theta(f \otimes g)$ and $\phi(\tau f) = \tau \phi(f)$, as are easily checked by reducing to the integral case.    
Now,  
\begin{align*}
 \theta( y'_a \otimes y'_b) & = \theta( \tau y_a \otimes \tau y_b )  = -\tau  \theta(y_a \otimes y_b ) \\ &  = -\tau  \phi( y_{[a,b]} ) = \phi(y'_{[a,b]}), \end{align*}
so $M_A^{\vee} \in \rep(\Al_t)$ indeed. Define the evaluation and coevaluation maps  $  M_A^{\vee} \otimes  M_A \to \unital$, $\unital \to M_A \otimes M_A^{\vee}$ in $\rep(\Al_t)$ as simply the usual   maps in $\rep(S_t)$; it needs only to be checked that they commute with the $y,y'$-actions.  We will do this for the coevaluation.

Let $u \in Q_1(M) \otimes_k Q_1(M)$.  
There is a map $w\colon\f{h} \to M \otimes M$ obtained
by $\f{h} \xrightarrow{(\id \otimes \coev)} \f{h} \otimes M \otimes M
\xrightarrow{(\id \otimes \tau )u} M \otimes M$.  Here $(\id \otimes \tau )u \in Q_1(M) \otimes_k Q_1(M)$, which is regarded as an element of $Q_1(M \otimes M)$. Here $\coev\colon \id \to M \otimes M$ is the coevaluation map in the tensor category $\rep(S_t)$.

\begin{lemma} The map $w\colon \f{h}   \to M\otimes M$ is dual to $m_{Q_1(M)}(u) \in Q_1(M)$.
\end{lemma}

\begin{proof} We need only prove this   in the integral case.  Suppose that $u = f \otimes g \in Q_1(M) \otimes_k Q_1(M)$.  Then this can be represented as two families $\{f_i\}, \{g_i\}$ of maps $M \to M$.  The map $\f{h} \to M \otimes M$ corresponds to a family of maps $\{h_i \colon M \to M\}$, which is  seen to be $f_i g_i $.  
\end{proof}

Recall that for a Hopf algebra $A$, the following diagram is commutative:
\[ \xymatrix{ A \ar[r]^{\Delta} \ar[rd]_{\eta} & A \otimes_k A \ar[r]^{\id \otimes S} & A \otimes_k A \ar[r]^m & A \\
& k \ar[rru]_{\epsilon} }. \]
There is a homomorphism $A \to Q_1(M \otimes M)$ from the action on $M_A \otimes M_A^{\vee}$. 
This is given by the composition  
\[ \xymatrix{ A  \ar[r]^{(\id \otimes S) \circ \Delta} &   A \otimes_k A \ar[r]^-{y \otimes y} & Q_1(M) \otimes_k Q_1(M) \ar[r]^-{\id \otimes \tau } & Q_1(M) \otimes_k Q_1(M) }. \]

Finally, we must show that the following diagram commutes:
\[ \xymatrix{
\f{h} \ar[d]^{\eta(a)} \ar[r]^-{\id \otimes \coev} &  \f{h} \otimes M \otimes M \ar[d]^{ (\id \otimes \tau ) \circ (y \otimes y) ( (\id \otimes S) \circ \Delta(a))}   \\
\unital \ar[r]^-{\coev} & M \otimes M } \]

The map obtained by going right and then down is equivalent to the map $\f{h} \otimes M \to M$ from $m((y \otimes y)(\id \otimes S)\circ \Delta(a)) = y_{\epsilon(\eta(a))}$.  Since $y$ is a unital homomorphism, this is equivalent to $\eta(a)(s \otimes \id_M)\colon \f{h} \otimes M \to M$.  The map $\f{h} \to \unital$ is, by definition, $\eta(a)$ times   $s\colon \f{h} \to \unital$, so the map $\f{h} \to M \otimes M$ obtained by going down and then right is equivalent to $\eta(a)(s \otimes \id_M)\colon \f{h} \otimes M \to M$.   The same is done to show that $ M_A^{\vee} \otimes M_A \to \unital$ is a morphism in $\rep(\Al_t)$. Similarly, the existence of left duals can be shown.
\end{proof}

\section{Wreath products with finite groups}

Fix a finite group $G$.  We recall the definition of the Knop categories $\rep(S_t \ltimes G^t)$ in nonintegral rank.

First suppose $n \in \mathbb{Z}_{\geq 0}$.  Let $U$ be a finite   $G$-set, i.e., one where the stabilizer of each element is trivial, or equivalently one isomorphic to a coproduct of copies of $G$.
To it, we associate a representation $\injr_G(U, \coprod_n G) \in \repord(S_n \ltimes G^n)$ by considering the permutation representation of $S_n \ltimes G^n$ on the set $\inj_G(U, \coprod_n G)$ of $G$-equivariant injections $U \to \coprod_t G$.  
Note that 
\[ \injr_G(U, \coprod_n G)  \otimes \injr_G(V, \coprod_n G) = \bigoplus_{C \in \rec_G(U,V)} \injr_G(C, \coprod_n G) \]
where $C$ ranges over the set $\rec_G(U,V)$ of ``$G$-recollements'' of $U,V$, where a $G$-recollement is defined as an equivariant analog of the usual notion: one has equivariant injections $U \to C, V \to C$ whose images cover $C$.
Next, when $n$ is large relative to $U,V$, the space
\( \hom_{S_n \ltimes G^n}(  \injr_G(U, \coprod_n G), \injr_G(V, \coprod_n G) ) \)
has a basis of $(C)_G, C \in \rec_G(U,V)$, where $(C)_G$ is defined in a similar manner as it was in Deligne's case.
Finally, one can find polynomials $Q^E_{C,D}(T)$ for $ C \in \rec_G(U,V),D \in \rec_G(V,W), E \in \rec_G(U,W)$ such that
\begin{equation} \label{compglaw} (C)_G\circ (D)_G = \sum_{ E \in \rec_G(U,W)} Q^E_{C,D}(n) (E)_G.\end{equation}

Therefore, one may define a $\mathbb{Z}[T]$-linear category $\rep'(S_T \ltimes G^T)$ whose objects are symbols $[U]_G$ for finite free $G$-sets $U$ and 
$\hom([U]_G,[V]_G) $ is the free module on symbols $(C)_G$ for $C \in \rec_G(U,V)$. Composition is defined as in \eqref{compglaw} with $T$ replacing $n$.
For  $R$ a ring and $t \in R$, one specializes to obtain $R$-linear categories $\rep'(S_t \ltimes G^t, R)$ and takes the pseudo-abelian envelope to obtain $\rep(S_t \ltimes G^t, R)$.  Henceforth, we take $R = k$ for $k$ an algebraically closed field of characteristic zero unless indicated otherwise.

It is now clear that Knop's categories fit into the framework developed above.
In particular, there is an $\A_{k}^1$-category $\rrep(S \ltimes G)$ whose fibers at $t \in k$ yield $\rep(S_t \ltimes G^t)$.  Moreover, there are full subcategories $\rrep(S \times G)^{(M,N)}$ of finite type that filter $\rrep(S \ltimes G)$, as before.

We shall investigate the relationship between Knop's categories and Etingof's categories of semidirect products.  In detail, fix a finite group $G$, and let  $A = k[G]$ be the group algebra.

\begin{proposition} There is a functor $F$ of $\A^1_{k}$-monoidal categories $\rrep(S \ltimes G) \to \rrep(S \ltimes A^{\otimes})$ interpolating the usual ones.
\end{proposition}

\begin{proof} 
We shall first associate to each object $[U]_G \in \rep'(S_T \ltimes G^T, k[T])$ an object of $\add\left( \rep'(S_T)\right)$ and appropriate $y$-morphisms.
Indeed, choose a set $u \subset U$ of representatives of the  orbits of $G$ in $U$, and note that for any $n$,
\( \injr_G( U , \coprod_n G) \simeq \bigoplus_{ G^{u}} \injr( u, I ).\)
So we define a map $\add \rep'(S_T \ltimes G^T) \to \add \rep'(S_T)$ by sending $[U]_G \to \bigoplus_{G^{u}} [u]$.  
The $y$-morphism can be described combinatorially. Given $g \in G$ and  $\bd{g}, \bd{g}' \in G^{u}$ and a recollement $C$ of $u, u, \one$, we must associate a coefficient in such a manner that it describes the homomorphism
\( y_g\colon \f{h} \otimes \bigoplus_{G^{u}} [u] \to  \bigoplus_{G^{u}} [u].\)
Suppose $C$ is given by $a\colon u \to C, b\colon u \to C$, and $c \in C$ corresponding to the image of $\one$.  Then $(C)$ occurs with coefficient $1$ if $a=b$, and $\bd{g}, \bd{g}'$ differ at one point corresponding to  $c$, by the value $g$.  Otherwise, $(C)$ occurs with coefficient zero. 
The functoriality of this map can be described combinatorially in terms of recollements in a similar way, and we omit the details.
The fact that this is a \emph{tensor} functor can be checked by specialization.  Indeed, it can be checked that the functor just defined actually interpolates the usual tensor functors.
\end{proof}

Our aim is to prove:
\begin{theorem} \label{tenseq} For $t \in k$ transcendental, the categories $\rep(S_t \ltimes G^t)$ and $\rep(S_t\ltimes A^{\otimes t})$ for $A = {k}[G]$ are tensor equivalent.
\end{theorem}

\begin{lemma} Let $S$ be a noetherian scheme.  Let $R, R', T, T'$ be 
schemes of finite type over $S$.  Suppose given a commutative diagram
\[ \begin{CD}
R @>>> R' \\
@VVV @VVV \\
T @>>> T' \end{CD}. \]
Then there is a constructible subset $C \subset T$ consisting of the $t \in T$ such that $R_{\b{t}} \to R_{\b{t'}}$ is bijective (where $t' \in T'$ is the image of $t$).
\end{lemma}
\begin{proof} There is a map $f\colon R \to T \times_{T'} R'$ of $T$-schemes.  Now, for $t \in T$,
\[ (T \times_{T'} R')_{\b{t}} \simeq \spec \overline{k(t)} \times_{T'} R' \simeq \spec \overline{k(t')} \times_{T'} R' \simeq (R')_{\b{t'}}.\]
Then the set $C$ consists of $t \in T$ with $f_{\b{t}}$ bijective.
Therefore, we can apply Lemma~\ref{bijisconst} to the map $f$ and get constructibility.  Note that $\overline{k(t')} \simeq \overline{k(t)}$ by the finite type hypotheses.  
\end{proof}

\begin{lemma} If $\f{C}$ is a finite-type $S$-category for $S$ noetherian and $C \subset \ob$ is constructible, there is a constructible set $D \subset \ob$ such that if $s \in S$, then the closed points of $D_{\b{s}}$ corresponds precisely to the objects of $\mathcal{C}_{\b{s}}$ isomorphic to objects in $C_{\b{s}}$.
\end{lemma}
\begin{proof}
There is a constructible subset $Iso \subset \mor$ such that $(Iso)_{\b{s}}$ consists precisely of the isomorphisms in $\mathcal{C}_{\b{s}}$.  This is  easy to see because $\comp$ is a morphism and the identities form a constructible subset.  Then, the lemma becomes clear.
\end{proof}

\begin{proposition} Let $\f{C}, \f{D}$ be finite-type $S$-categories for $S$ a noetherian scheme and $F\colon \f{C} \to \f{D}$ a functor over $S$.  Then the set of $s \in S$ with $F_{\b{s}}\colon \mathcal{C}_{\b{s}} \to \mathcal{D}_{\b{s}}$ fully faithful (resp. essentially surjective) is constructible.
\end{proposition}
\begin{proof} 
Let $\ob_C, \mor_C, \ob_D, \mor_D$ be the object and morphism schemes of $\f{C}, \f{D}$. 
Denote the induced morphisms of schemes $\ob_C \to \ob_D, \mor_C \to \mor_D$ by $F_{\ob}, F_{\mor}$, respectively.
 Let $T \subset \ob_C$ be the set of points $t$ such that the morphism $F_{\mor}\colon (\mor_C)_{\b{t}} \to (\mor_D)_{\b{F_{\ob}(t)}}$ is bijective.  Let $S_1 \subset S$ be defined as $S_1 = S- f(\ob_C - T)$, if $f\colon \ob_C \to S$ is the projection.  Then $S_1$ is constructible and corresponds to the set of $s \in S$ such that $F_{\b{s}}\colon \mathcal{C}_{\b{s}} \to \mathcal{D}_{\b{s}}$ is fully faithful. 

Consider $U_1 = F_{\ob}(\ob_C)$.  Let $U$ be constructed from $U_1$ as in the previous lemma.  Let $S_2 = S - f( \ob_D - U)$.  Then $S_2$ corresponds to points $s$ with $F_{\b{s}}$ essentially surjective. 
\end{proof}

\begin{proof}[Proof of Theorem~\ref{tenseq}] The internal functor $F$ as defined previously induces tensor functors $F_t\colon \rep(S_t \ltimes G^{ t}) \to \rep(\Al_t)$.   
Then the induced functor $\rep(S_t \ltimes G^t)^{(N)} \to \rep(\Al_t)^{(N)}$ is an equivalence when $t \in \mathbb{N}$ and $t \gg N$. More precisely, we can see that 
the functor is  faithful, and is full when considered as a functor 
from a subcategory of $\rep(S_t \ltimes G^t)^{(N, N')} \to \rep(\Al_t)^{(N, N'')}$, where $N'' \gg N'$ depends only on $N', N$.  This last fact is thus true generically.  The theorem follows.
\end{proof}

\bibliographystyle{abbrv}
\bibliography{Siemenspaper}

\end{document}